\documentclass[3p]{elsarticle}

\usepackage{lineno,hyperref}
\usepackage[english]{babel}
\usepackage{amssymb}
\usepackage{multicol}
\usepackage{amsfonts}
\usepackage{algpseudocode} 
\usepackage{algorithm} 
\usepackage{graphicx}
\usepackage{caption}
\usepackage{bm}
\usepackage{bbm} 
\usepackage{amsthm} 
\usepackage{amsmath}
\usepackage{color}
\usepackage{picture}
\usepackage{tikz}
\usepackage{subfig}
\usepackage{adjustbox} 

\newcommand{\vect}[1]{\boldsymbol{#1}}

\renewcommand{\d}{\ \mathrm{d}}
  
\newcommand{\z}{\phantom{0}}
\renewcommand{\d}{\mathrm{d}}

 \usepackage{cases}

\newtheorem{ass}{Assumption}
\newtheorem{thm}{Theorem} 

\newtheorem{rmk}{Remark}

\newtheorem{lem}{Lemma}

\numberwithin{equation}{section}

\renewcommand{\d}{\mathrm{d}}
\newcommand{\dt}{\mathrm{dt}}

\newcommand{\kron}{\otimes}

\def\R1{\color{black}} 
\def\B{\color{black}}

\definecolor{ao(english)}{rgb}{0.1, 0.5, 0.1}
\definecolor{cobalt}{rgb}{0.0, 0.28, 0.67}

\renewcommand{\S}{\mathcal{S}}

\newcommand{\vecx}{\boldsymbol{x}}

\newcommand{\vertiii}[1]{{\left\vert\kern-0.25ex\left\vert\kern-0.25ex\left\vert #1 
    \right\vert\kern-0.25ex\right\vert\kern-0.25ex\right\vert}}
 
\DeclareMathOperator*{\argmin}{arg\,min}

\newcommand*\mathinhead[2]{\texorpdfstring{${#1}$}{#2}}

\newcommand{\strain}[1]{\adjustbox{max width=\textwidth}{$\displaystyle#1$}}

\begin{document}

\begin{frontmatter}

\title{\R1 Space-time least squares approximation for Schr{\"o}dinger equation and efficient solver \B}
\author[imati]{Andrea Bressan}\ead{andrea.bressan@imati.cnr.it}
\author[mat]{Alen Kushova\corref{cor1}}\ead{alen.kushova01@universitadipavia.it}
\author[mat,imati]{Giancarlo Sangalli} \ead{giancarlo.sangalli@unipv.it}
\author[mat]{Mattia Tani} \ead{mattia.tani@unipv.it}

\address[mat]{Dipartimento di Matematica ``F. Casorati", Universit\`{a} di Pavia, Via A. Ferrata, 5, 27100 Pavia, Italy.}

\address[imati]{Istituto di Matematica Applicata e Tecnologie Informatiche, ``E. Magenes" del CNR, Via A. Ferrata, 1, 27100 Pavia, Italy.}
 \cortext[cor1]{Corresponding author}
\begin{abstract}
In this work we present a space-time least squares isogeometric discretization 
of the Schr{\"o}dinger equation and propose a preconditioner for the arising 
linear system in the parametric domain. Exploiting the tensor product structure of the 
basis functions, the preconditioner is written as the sum of Kronecker products of 
matrices. Thanks to an extension to the classical Fast Diagonalization method, the application 
of the preconditioner is efficient and robust w.r.t. the polynomial degree of the spline space. The time
required for the application is almost proportional to the number of degrees-of-freedom, for a serial execution.
\end{abstract}
\end{frontmatter}



\section{Introduction}
Space-time finite element methods originated in the papers
\cite{fried1969finite, bruch1974transient, oden1969general, oden1969general2}, 
where standard finite elements are ascribed an extra dimension 
for the time  and, typically,  adopt a  discontinuous approximation 
in time, since this produces a time marching algorithm with a traditional 
step-by-step format (see e.g. \cite{shakib1991new}).  
Over the years, the theory of space-time methods has been developed mainly for evolutionary equations of the parabolic type
and hyperbolic type, whereas, for quantum mechanics, and more precisely for Schr{\"o}dinger's equation, there are few contributions 
and the methods are still in development.

To our knowledge, one of the first works concerning space-time variational 
formulations for the (nonlinear) Schr{\"o}dinger equation, is \cite{karakashian1998space}, 
in which Karakashian and Makridakis proposed a space–time method combining a conforming 
Galerkin discretization in space and an upwind DG time-stepping. This method reduces 
to a Radau IIA Runge-Kutta time discretization in the case of constant potentials.
In \cite{DEMKOWICZ}, for the linear Schr{\"o}dinger equation the authors propose 
two variational formulations that are proved to be well posed: a strong formulation, 
with no relaxation of the original equation, and an ultraweak formulation, that transfers 
all derivatives onto test functions. The proposed discretization for the ultraweak form 
is based on a discontinuous Petrov-Galerkin (DPG) method, that addresses optimal stability, 
and quasi-optimal error rates in $L^2$-norm. 
In \cite{gomez2022space} a space–time ultraweak Trefftz discontinuous Galerkin (DG) method 
for the Schr{\"o}dinger equation with piecewise-constant potential is proposed and analyzed, 
proving well-posedness and stability of the method, and optimal high-order $h$-convergence 
error estimates in a skeleton norm, for the one and two dimensional cases. 
Recently, in \cite{hain2022ultra}, Hain and Urban proposed a space–time ultraweak variational formulation 
with optimal inf-sup constant. The formulation in \cite{hain2022ultra} is
related to the ultraweak DPG method in \cite{DEMKOWICZ}, but differs in the choice 
of the test and trial spaces. Hain and Urban first fix a conforming test space, and then 
construct an optimal trial space, while Demkowicz et al. first constructs a trial space
and then a suitable test space. The discretization proposed in \cite{hain2022ultra} uses high order B-splines 
with maximum regularity and can be extended to the Isogeometric Analysis (IgA) framework.

Introduced in \cite{Hughes2005}, see also the book \cite{Cottrell2009}, 
IgA, is an evolution of the classical finite element methods. In IgA, both the approximation of the solution of the 
partial differential equation that models the problem, and the representation of the computational domain, 
are accomplished using B-spline functions, or their generalizations (NURBS). 
This is meant to simplify the interoperability between computer aided design and numerical simulations.
IgA also benefits  from the approximation properties of splines,  whose high-continuity   
yields  higher accuracy when compared to $C^0$ piecewise polynomials, see e.g.,
\cite{Evans_Bazilevs_Babuska_Hughes,bressan2018approximation,Sangalli2018}.

In this paper we focus on the linear time dependent Schr{\"o}dinger equation without potential. 
Starting from the well posed space-time strong formulations in \cite{DEMKOWICZ}, we derive a well posed 
space-time isogeometric Petrov-Galerkin discretization, that is essentially a Galerkin approximation of the space-time
least squares variational formulation of the model problem. The matrix associated to the discrete linear system can 
be written as sum of Kronecker products, and has the same structure of the one arising from \cite{hain2022ultra}.
The main contribution of this paper, is the development of a stable preconditioner that leads to 
a fast solver for the problem modeled in the parametric domain. As it was done in \cite{Montardini2018space,loli2020efficient} for parabolic problems, 
our preconditioner exploits the Kronecker structure of the linear system, and makes use of Fast Diagonalization method (FD) \cite{Lynch1964}.
In this work, FD is applied among the space direction only. Although, the computational cost of the setup
of the resulting preconditioner is $O(N_{dof})$ FLOating-Point operations (FLOPs), 
while its  application  is  $O(N_{dof}^{1 + 1/d})$ FLOPs, where $d$ is the number of spatial 
dimensions and  $N_{dof}$ denotes the total number of degrees-of-freedom (assuming, for simplicity,  
to have  the same number of  degrees-of-freedom in time and in each spatial direction).
We remark that global space-time methods, in principle,  facilitate the full 
parallelization of the solver, see \cite{dorao2007parallel,Gander2015,kvarving2011fast}. 

The outline of the paper is as follows. The model problem is introduced in Section 2. 
In Section 3 we present the basics of B-splines based IgA and the main properties of the Kronecker product operation. 
The isogeometric least squares discretization is introduced in Section 4, while in Section 5 we define the preconditioner for the parametric domain
and we discuss its application. We present the numerical results assessing the performance of the proposed 
preconditioner in Section 6. Finally, in the last section we draw some conclusions and we highlight 
some future research directions.

\section{Model problem}
We consider a bounded domain ${\Omega} \subset \mathbb{R}^d$, usually $d=1,2,3$, with Lipschitz boundary, and a time 
interval $(0,T)$, where $T > 0 $ is the final time. The space-time domain is denoted by $\mathcal{Q} := (0,T) \times {\Omega} $. 
Assuming Dirichlet boundary conditions, denote by $\Gamma_D := (0,T) \times \partial{\Omega} $
the Dirichlet boundary of the space-time cylinder $\mathcal{Q}$, while $\mathcal{Q}_0 = \{0\}\times {\Omega} $ is the initial side.
Our model problem is the Schr{\"o}dinger equation with homogeneous boundary and initial conditions:   we look for a solution $u$ such that 
\begin{equation}
\label{eq:problem}
	\left\{
	\begin{array}{rcllc}
		 \mathrm{i}\partial_t u -  \nu \Delta u  & = & f & \mbox{in } & \mathcal{Q},\\[1pt]
		 u  & = & 0 & \mbox{on } &\Gamma_D,\\[1pt]
		 u & = & 0 & \mbox{in } &\mathcal{Q}_0,
	\end{array}
	\right.
\end{equation}
where $\mathrm{i}$ is the imaginary unit and $\nu > 0$ is a constant coefficient 
usually depending on Planck's constant $\mathfrak{\hbar}$ and the mass of the modeled physical particle.   
We assume that  $f\in L^2(\mathcal{Q})$ and denote by $\mathbb{S} := \mathrm{i}\partial_t -\nu \Delta $ 
the Schr{\"o}dinger operator, $\mathbb{S}^*$ its formal adjoint, and $(\cdot,\cdot)$ the complex scalar 
product in $L^2(\mathcal{Q})$.

\subsection{Space-time variational formulation} 
\label{sec:problem}
Let us introduce the Hilbert spaces
\[\mathcal{V} \strain{
  :=\left\{ v \in L^2(\mathcal{Q}) : \mathbb{S}v \in L^2(\mathcal{Q}) \text{ and } (\mathbb{S}^*w,v) - (w,\mathbb{S}v) = 0 
                      \,\,\forall w \in \mathcal{C}^{\infty}_{0}(\mathbb{R}^{d+1}) : w |_{\Gamma_D \cup (\{T\}\times \Omega) } = 0 \right\},
  }
\]
\[
    \mathcal{W}:=  L^2(\mathcal{Q}),
  \]
endowed with the following norms  
\begin{equation*}
  \|v\|_{ \mathcal{V}}^2:=   
  \|v\|^2_{ L^2(\mathcal{Q})}+ \|\mathbb{S}v\|^2_{ L^2(\mathcal{Q})}  
  \ \text{ and }\ 
  \|w\|_{ \mathcal{W}}:=   \|w\|_{ L^2(\mathcal{Q})},
\end{equation*}
respectively.   Then, the space-time variational formulation of \eqref{eq:problem} reads:  
\begin{equation}
\label{eq:var_for}
\text{Find } u \in \mathcal{V} \text{ such that }  \mathcal{A}( {u},v) =
\mathcal{F} (v)  \quad \, \forall v \in  \mathcal{W},
\end{equation}  
where the sesquilinear form $\mathcal{A}(\cdot,\cdot)$ and the linear form $\mathcal{F}(\cdot)$ are defined $ \forall v\in\mathcal{V} \text{ and } \forall w\in \mathcal{W}$ as
\[
 \mathcal{A}(v,w) := \int_{\Omega}\int_{0}^T  \left(\mathbb{S} v \right) \overline{w} \,\dt \, \d\Omega 
\quad \text{and} \quad
\mathcal{F}(w)  := \int_{\Omega}\int_{0}^T f\,   \overline{w} \,\dt \, \d\Omega.
\] 
The well-posedness  of the variational formulation above, for $d=1$, is in \cite{DEMKOWICZ}, but the generalization to $d>1$ is straightforward{, see the Appendix of this paper}.    

\subsection{Extensions}
The previous setting can be generalized  to    non-homogeneous initial and boundary conditions.
For example, suppose that in \eqref{eq:problem} we have the initial condition $u=u_0$ in $\mathcal{Q}_0$ with $u_0\in L^2(\Omega).$  Then, we consider a lifting  $\underline{u}_0$ of $u_0$ such that $\underline{u}_0\in L^2(\mathcal{Q})$, see e.g. \cite{Evans2010book}. 
Finally, we split the solution $u$ as $u=\underline{u}+\underline{u}_0$, where $\underline{u}\in\mathcal{X}$ is the solution of the following Schr{\"o}dinger equation with homogeneous initial and boundary conditions:
 \begin{equation*} 
	\left\{
	\begin{array}{rcllc}
    \mathrm{i}\partial_t \underline{u} -  \nu \Delta \underline{u}  & = & \underline{f} & \mbox{in } & \mathcal{Q},\\[1pt]
    \underline{u}  & = & 0 & \mbox{on } &\Gamma_D,\\[1pt]
    \underline{u} & = & 0 & \mbox{in } &\mathcal{Q}_0,
 \end{array}
	\right.
	 \end{equation*}
	 where $\underline{f}:=f-\mathbb{S}\underline{u}_0$.

\section{Isogeometric framework and preliminaries}
\subsection{B-Splines}

Given $m$ and $p$ two positive integers, a 
 knot vector in $[0,1]$ is   a sequence of non-decreasing points
 $\Xi:=\left\{ 0=\xi_1 \leq \dots \leq \xi_{m+p+1}=1\right\}$. 
 We consider open knot vectors, i.e.  we set $\xi_1=\dots=\xi_{p+1}=0$
 and $\xi_{m}=\dots=\xi_{m+p+1}=1$, and denote by $Z= \{\zeta_1,\dots,\zeta_{r}\}$ the vector of breakpoints, that is the vector of knots without repetition.

Univariate B-splines $\widehat{b}_{i,p}:(0,1)\rightarrow \mathbb{R}$ are  piecewise polynomials defined according to   Cox-De Boor recursion formulas (see  \cite{DeBoor2001}).
The univariate spline space is defined as
\begin{equation*}
\widehat{\S}_h^p : = \mathrm{span}\{\widehat{b}_{i,p}\}_{i = 1}^m,
\end{equation*}
where $h$ denotes the mesh-size, i.e. $ h:=\max\{ |\xi_{i+1}-\xi_i| \ | \ i=1,\dots,m+p \}$.
The interior knot multiplicity
  influences the smoothness of the B-splines at
the knots (see \cite{DeBoor2001}). For more details on B-splines properties  and their use in IgA we refer to  \cite{Cottrell2009}.

Multivariate B-splines are defined as tensor product of univariate
B-splines.  
We consider  functions that depend on $d$ spatial variables  and the time variable. Therefore, given positive integers $m_l, p_l$ for $l=1,\dots,d$ and $m_t,p_t$, we  introduce $d+1 $
univariate knot vectors $\Xi_l:=\left\{ \xi_{l,1} \leq \dots \leq
  \xi_{l,m_l+p_l+1}\right\}$ , with associated breakpoints $Z_l= \{\zeta_{l,1},\dots,\zeta_{l,r_l}\}$,  for $l=1,\ldots, d$ and   $\Xi_t:=\left\{
  \xi_{t,1} \leq \dots \leq \xi_{t,m_t+p_t+1}\right\}$ with  $Z_t= \{\zeta_{t,1},\dots,\zeta_{t,r_t}\}$.
  Let $h_l$ be the mesh-size associated to the knot vector $\Xi_l$ for $l=1,\dots,d$, let $h_s:=\max\{h_l\ | \ l=1,\dots, d\}$ be the maximal mesh-size in all spatial knot vectors and let $h_t$ be the mesh-size of the time knot vector. 
  Let also $\boldsymbol{p}$ be the vector that contains the degree indexes, i.e. $\boldsymbol{p} :=(p_t,\boldsymbol{p}_s)$, 
  where $\boldsymbol{p}_s:= (p_1,\dots,p_d )$. For simplicity, we assume to have the same polynomial degree in all spatial directions, i.e., with abuse of notations, we set  $p_1=\dots=p_d=:p_s$, but the general case is similar.
 
We assume that the following local quasi-uniformity of the knot vectors holds.

\begin{ass}
\label{ass:quasi-uniformity}
There exists  $\theta \geq 1 $, independent of $h_s$ and $h_t$, such that $ \theta^{-1} \leq \zeta_{l,i} / \zeta_{l,i+1} \leq
\theta$  for $i = 1,\dots,r_l$, $l =1,\dots,  d$   and  $ \theta^{-1} \leq \zeta_{t,i}/\zeta_{t,i+1} \leq \theta$ for $i = 1,\dots,r_t$.
\end{ass}

The multivariate B-splines are defined as
\begin{equation*} 
\widehat{B}_{\vect{i},\vect{p}}(\tau, \vect{\eta}) : =
\widehat{b}_{i_t,p_t}(\tau) \widehat{B}_{\vect{i_s}, \vect{p}_s}(\vect{\eta}),
\end{equation*}
 where 
 \begin{equation}
   \label{eq:tens-prod-space-parametric}
   \widehat{B}_{\vect{i_s},\vect{p}_s}(\vect{\eta}):=\widehat{b}_{i_1,p_s}(\eta_1) \ldots \widehat{b}_{i_d,p_s}(\eta_d),
 \end{equation}
  $\vect{i_s}:=(i_1,\dots,i_d)$, $\vect{i}:=(i_t,\vect{i_s})$  and  $\vect{\eta} = (\eta_1, \ldots, \eta_d)$.  
The  corresponding spline space  is defined as
\begin{equation*}
\widehat{\S} ^{\vect{p}}_{ {h}  }  := \mathrm{span}\left\{\widehat{B}_{\vect{i}, \vect{p}} \ \middle| \  i_t=1,\dots,m_t; \ i_l = 1,\dots, m_l \text{ for } l=1,\dots,d \right\},
\end{equation*} 
  where $h:=\max\{h_t, h_s\}$. 
We have that
$\widehat{\S} ^{\vect{p}}_{ {h}}  = \widehat{\S} ^{p_t}_{h_t} \otimes \widehat{\S} ^{ \vect{p}_s}_{ {h}_s}, $ 
where   
\[\widehat{\S} ^{\vect{p}_s}_{h_s} 
  := \mathrm{span}\left\{
  \widehat{B}_{\vect{i_s},\vect{p}_s} \ \middle|  \ i_l =
  1,\dots, m_l; l=1,\dots,d  \right\}
\] is the space of tensor-product splines on $\widehat{\Omega}:=(0,1)^d$.  Finally, we make the following regularity assumption.
\begin{ass} 
\label{ass:knot_mult}
We assume that $p_t\geq 1$, $p_s\geq 2$ and that
$\widehat{\S} ^{{p}_t}_{h_t} \subset C^0\left((0,1)\right)$  
and  
$\widehat{\S} ^{\vect{p}_s}_{h_s} \subset C^1(\widehat{\Omega}  )$. 
\end{ass}


\subsection{Isogeometric spaces}
\label{sec:iso_space}
The space-time computational domain that we consider is $(0,T) \times \Omega$, where $T>0$ is the final time and $\Omega\subset\mathbb{R}^d$ is the space domain. 
{ The choice of considering the time as first variable will be clarified in Section \ref{sec:application_of_prec}}.  
The following assumptions asserts the regularity of the parametrization. 
 
\begin{ass}
\label{ass: regular-single-patch-domain}
 We assume  that  $\Omega$  is parametrized by  $\vect{F}: \widehat{\Omega} \rightarrow {\Omega} $,
 with  $\vect{F}\in  {\left[\widehat{\mathcal{S}}^{\vect{p}_s}_{{h}_s}\right]^d} $ on the closure of $\widehat{\Omega}$. Moreover, we assume that 
 $\vect{F}^{-1}$ has piecewise bounded derivatives of any order.
\end{ass}

 Denote by  $\vecx=(x_1,\dots,x_d):=  \vect{F}(\vect{\eta})$ and
 $t:=T\tau$. Then the space-time domain  is given by  the parametrization
 $\vect{G}:(0,1) \times \widehat{\Omega}\rightarrow (0,T) \times \Omega$, 
 such that $ \vect{G}(\tau, \vect{\eta}):=(T\tau, \vect{F}(\vect{\eta}))=(t,\vecx).$
 
 We introduce the spline space with initial and boundary
 conditions, in parametric coordinates, as  
\begin{equation*}
\widehat{\mathcal{X}}_{h,0}:=\left\{ \widehat{v}_h\in \widehat{\mathcal{S}}^{\vect{p}}_h \ \middle| \ \widehat{v}_h = 0 \text{ on } \{0\} \times \widehat{\Omega} \text{ and } \widehat{v}_h = 0 \text{ on } (0,1) \times \partial\widehat{\Omega} \right\}.
\end{equation*}
 Notice that  
  $ \widehat{\mathcal{X}}_{h,0} =   \widehat{\mathcal{X}}_{t,h_t,0}   \otimes \widehat{\mathcal{X}}_{s,h_s}     $,  where 
 \begin{subequations}
 \begin{align*}
  \widehat{\mathcal{X}}_{t,h_t,0} & := \left\{ \widehat{w}_h\in \widehat{\mathcal{S}}^{ p_t}_{h_t} \ \middle|  \ \widehat{w}_h( 0)=0 \right\}  \   = \ \text{span}\left\{ \widehat{b}_{i_t,p_t} \ \middle| \ i_t = 2,\dots , m_t\ \right\},\\
  \widehat{\mathcal{X}}_{s,h_s}   & := \left\{ \widehat{w}_h\in \widehat{\mathcal{S}}^{\vect{p}_s}_{h_s}    \ \middle| \ \widehat{w}_h = 0 \text{ on } \partial\widehat{\Omega}  \right\}\  \\
  &  \ = \ \text{span}\left\{ \widehat{b}_{i_1,p_s}\dots\widehat{b}_{i_d,p_s} \ \middle| \ i_l = 2,\dots , m_l-1; \ l=1,\dots,d\ \right\}.
 \end{align*}
 \end{subequations}
Analogously, the spline space with final and boundary conditions, in parametric coordinates, is 
\begin{align*} 
  \widehat{\mathcal{X}}_{h,T} 
  &:= \left\{ \widehat{v}_h\in \widehat{\mathcal{S}}^{\vect{p}}_h \ \middle| \ \widehat{v}_h = 0 \text{ on } \{T\} \times \widehat{\Omega} \text{ and } \widehat{v}_h = 0 \text{ on } (0,1) \times \partial\widehat{\Omega} \right\} \\
  &=\widehat{\mathcal{X}}_{t,h_t,T}   \otimes \widehat{\mathcal{X}}_{s,h_s} 
\end{align*}
where 
$$
\widehat{\mathcal{X}}_{t,h_t,T}:= \left\{ \widehat{w}_h\in \widehat{\mathcal{S}}^{ p_t}_{h_t} \ \middle|  \ \widehat{w}_h( T)=0 \right\}  \   = \ \text{span}\left\{ \widehat{b}_{i_t,p_t} \ \middle| \ i_t = 1,\dots , m_t-1\ \right\}.
$$

More precisely, by reordering the basis functions, i.e. for the space with initial and boundary homogeneous conditions, we write  
\begin{align*}
  \  \widehat{\mathcal{X}}_{t,h_t,0}  &  = \ \text{span}\left\{ \widehat{b}_{i,p_t} \ \middle| \ i = 1,\dots , n_t\ \right\},\\
  \widehat{\mathcal{X}}_{s,h_s} &   = \ \text{span}\left\{ \widehat{b}_{i_1,p_s}\dots\widehat{b}_{i_d,p_s} \ \middle| \ i_l = 1,\dots , n_{s,l}; \ l=1,\dots,d\ \right\}\\
  & \ =\text{span}\left\{ \widehat{B}_{i, \vect{p}_s} \ \middle|\ i =1,\dots , N_s   \ \right\},
\end{align*}
  and then
\begin{equation}
\widehat{\mathcal{X}}_{h,0}=\text{span}
\left\{ \widehat{B}_{{i}, \vect{p}} \ \middle|\ i=1,\dots,N_{dof} \right\},
\label{eq:all_basis}
\end{equation}
where we defined $n_t:=m_t-1$, $n_{s,l}:= m_l-2 $ for $l=1,\dots,d$, $N_s:=\prod_{l=1}^dn_{s,l}$ and $ N_{dof}:=n_t N_s$.  
We can proceed analogously with the space with final conditions. 

Finally, the isogeometric space we consider is the isoparametric push-forward of \eqref{eq:all_basis} through the geometric map $\vect{G}$, i.e.
\begin{equation}
\mathcal{X}_{h,0} := \text{span}\left\{  B_{i, \vect{p}}:=\widehat{B}_{i, \vect{p}}\circ \vect{G}^{-1} \ \middle| \ i=1,\dots , N_{dof}   \right\}.
\label{eq:disc_space}
\end{equation}
Notice that   
 $\mathcal{X}_{h,0}= \mathcal{X}_{t,h_t,0} \otimes \mathcal{X}_{s,h_s} $,
   where 
\[ 
  \mathcal{X}_{t,h_t,0}   :=\text{span}\left\{  {b}_{i,p_t}:= \widehat{b}_{i,p_t}( \cdot /T) \ \middle| \ i=1,\dots,n_t \right\},
  \]
and
\[ 
  \mathcal{X}_{s,h_s}    :=\text{span}\left\{ {B}_{i, \vect{p}_s}:= \widehat{B}_{i, \vect{p}_s}\circ \vect{F}^{-1} \ \middle| \ i=1,\dots,N_s \right\}.
\]
Analogously, we define $\mathcal{X}_{h,T}$, the isogeometric space with homogeneous Dirichlet and final conditions.


\subsection{Kronecker product}
 
The Kronecker product of two matrices $\mathbf{C}\in\mathbb{C}^{n_1\times n_2}$ and $\mathbf{D}\in\mathbb{C}^{n_3\times n_4}$ is defined as
\[
\mathbf{C}\otimes \mathbf{D}:=\begin{bmatrix}
[\mathbf{C}]_{1,1}\mathbf{D}  & \dots& [\mathbf{C}]_{1,n_2}\mathbf{D}\\
\vdots& \ddots &\vdots\\
[\mathbf{C}]_{n_1, 1}\mathbf{D}& \dots & [\mathbf{C}]_{n_1, n_2}\mathbf{D}
\end{bmatrix}\in \mathbb{C}^{n_1n_3\times n_2 n_4},
\]
where   $[\mathbf{C}]_{i,j}$ denotes the $ij$-th entry of the matrix $\mathbf{C}$. 
For extensions and   properties of the Kronecker product   we refer to \cite{Kolda2009}.
In particular, when   a  matrix   has a Kronecker product structure, the matrix-vector product   can be efficiently computed. For this
purpose, for $m=1,\dots,d+1$  we  introduce the $m$-mode product   of a tensor $\mathfrak{X}\in\mathbb{C}^{n_1\times\dots\times n_{d+1}}$ with a matrix $\mathbf{J}\in\mathbb{C}^{\ell\times n_m}$,   that we denote by  $\mathfrak{X}\times_m \mathbf{J}$.  This is 
  a tensor of size $n_1\times\dots\times n_{m-1}\times \ell \times n_{m+1}\times \dots n_{d+1}$,   whose elements are  defined as
\[\left[ \mathfrak{X}\times_m \mathbf{J} \right]_{i_1, \dots, i_{d+1}} := \sum_{j=1}^{n_m} [\mathfrak{X}]_{i_1,\dots, i_{m-1},j,i_{m+1},\dots,i_{d+1}}[\mathbf{J}]_{i_m,j }.\]
Then, given $\mathbf{J}_i\in\mathbb{C}^{\ell_i\times n_i}$ for $i=1,\dots, d+1$, it holds
\begin{equation}
\label{eq:kron_vec_multi}
\left(\mathbf{J}_{d+1}\otimes\dots\otimes \mathbf{J}_1\right)\text{vec}\left(\mathfrak{X}\right)=\text{vec}\left(\mathfrak{X}\times_1 \mathbf{J}_1\times_2 \dots \times_{d+1}\mathbf{J}_{d+1} \right),
\end{equation} 
 where the vectorization operator ``vec" applied to a tensor stacks its entries  into a column vector as
 \[ [\text{vec}(\mathfrak{X})]_{j}=[\mathfrak{X}]_{i_1,\dots,i_{d+1}} \text{ for }   i_l=1,\dots,n_{l} \text{ and for } l=1,\dots,d+1,   \] 
where  $j:=i_1+\sum_{k=2}^{d+1}\left[(i_k-1)\Pi_{l=1}^{k-1}n_l\right]$.

\section{Space-time discretization of the Schr{\"o}dinger equation}
 
\subsection{Instability of the space-time Galerkin method}
Let $\mathcal{V}_h := \mathcal{X}_{h,0}$ be the isogeometric space
defined in  \eqref{eq:disc_space} endowed with the $\|\cdot\|_{\mathcal{V}}$-norm, and choose
$\mathcal{W}_h := \mathcal{X}_{h,0}$ endowed with the $\|\cdot\|_{\mathcal{W}}$-norm.
Consider the following Galerkin method for \eqref{eq:var_for}:
\begin{equation}
\label{eq:discrete-system}
 \text{Find }   u_h\in \mathcal{V}_h \text{ such that } \mathcal{A}(u_h, w_h) = \mathcal{F}(w_h) \quad \, \forall w_h\in  \mathcal{W}_h.
\end{equation} 
The stability and the  well-posedness of formulation \eqref{eq:discrete-system} are not guaranteed,
indeed the inf-sup constant
$$
{\alpha}_h := \inf_{v_h \in \mathcal{V}_h} \sup_{w_h \in \mathcal{W}_h} \dfrac{|\mathcal{A}(v_h,w_h)|}{\|v_h\|_{\mathcal{V}}\|w_h\|_{\mathcal{W}}}>0,
$$
depends on the mesh size $h$ and degenerates under mesh refinement, as shown in Figure \ref{fig:galerkin_infsup_test}. 
\begin{figure}
  \centering
  \subfloat[]
  [Instability of space-time Galerkin method. \label{fig:galerkin_infsup_test}]
  {\includegraphics[width = 0.5\textwidth]{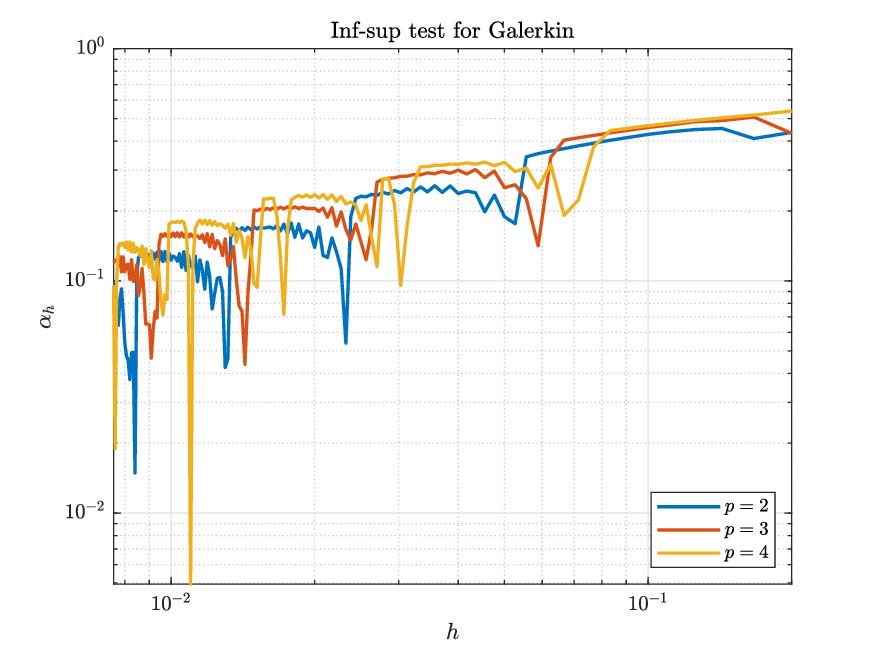}}
  \hfill
  \subfloat[]
  [Stability of space-time least squares method. \label{fig:least_squares_infsup_test}]
  {\includegraphics[width = 0.5\textwidth]{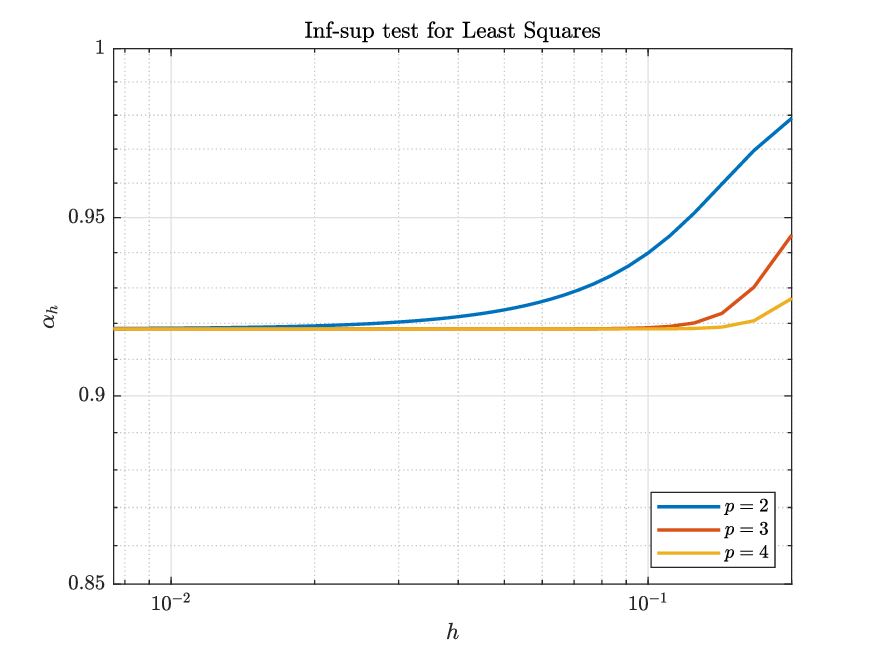}}
  \caption{Inf-sup test for the space-time discretizations.}
  \label{fig:infsup_test}
\end{figure}

\subsection{Least squares space-time method}
In order to retrieve a well posed space-time discretization to \eqref{eq:var_for},
notice that, given the quadratic functional $\mathcal{J}: \mathcal{V} \to \mathbb{R}$, defined as 
\begin{equation*}
    \mathcal{J}(u) := \dfrac{1}{2} \|\mathbb{S}u - f\|^2_{L^2(\mathcal{Q})},
\end{equation*}
we can write the least squares space-time formulation of \eqref{eq:problem}: find $u \in \mathcal{V}$ such that 
\begin{equation*}\label{eq:least_square_space_time_problem}
    u = \argmin_{v \in \mathcal{V}} \mathcal{J}(v),
\end{equation*}
and its Euler-Lagrange equations are 
\begin{equation*}
 \left( \mathbb{S}u, \mathbb{S}v \right) = \left( f, \mathbb{S}v \right),\quad \forall v \in \mathcal{V}.
\end{equation*}

This suggests to consider the following least-square discretization method for \eqref{eq:var_for}.
Let $\mathcal{V}_h := \mathcal{X}_{h,0}$ be the isogeometric space defined in  \eqref{eq:disc_space} 
endowed with the $\|\cdot\|_{\mathcal{V}}$-norm, and choose $\mathcal{W}_h := \mathbb{S}(\mathcal{V}_h)$ 
endowed with the $\|\cdot\|_{\mathcal{W}}$-norm. Consider the following Petrov-Galerkin approximation 
method for \ref{eq:var_for}:
\begin{equation}
 \label{eq:ls-discrete-system}
  \text{Find }   u_h\in \mathcal{V}_h \text{ such that } \mathcal{A}(u_h, w_h) = \mathcal{F}(w_h) \quad \, \forall w_h\in  \mathcal{W}_h.
\end{equation} 
Notice that, $\mathbb{S}$ is a bijection between the two discrete spaces, which means, for any $h>0$ it exists the inf-sup constant $\underline{\alpha}_h>0$. 
Moreover, for this discretization, the inf-sup $\alpha_h$ is uniformly bounded from below by a positive constant $\alpha>0$, as it is investigated numerically in Figure \ref{fig:least_squares_infsup_test}.
\begin{thm}\label{thm:quasi-optimality} 
  There exists a unique solution $u_h\in \mathcal{V}_h$ to the discrete
  problem \eqref{eq:ls-discrete-system}. Moreover, it holds 
  \begin{equation*} 
  \|u-u_h\|_{\mathcal{V}}\leq \dfrac{1}{\alpha} \inf_{v_h\in \mathcal{V}_h}\|u-v_h \|_{\mathcal{V}},
  \end{equation*}
  where $u\in\mathcal{V}$ is the solution of \eqref{eq:var_for}.
\end{thm}
We have then the following a-priori estimate for  $h$-refinement. 
\begin{thm}
  Let $q_t$ and $q_s$ be two integers, such that $ q_t\geq1$ and $q_s \geq 2$. 
  If  $u \in \mathcal{V} \cap \left( H^{q_t} (0,T) \otimes H^{2}(\Omega) \right) \cap \left( H^{1} (0,T) \otimes H^{q_s}(\Omega) \right)$ 
  is the solution of \eqref{eq:var_for} and $u_h \in \mathcal{V}_h$ is the solution of \eqref{eq:ls-discrete-system},
  then it holds
  \begin{equation} 
    \label{eq:a-priori-error-bound}
    \|u-u_h\|_{\mathcal{V}}\leq C \left(h_t^{k_t-1}\| u \|_{ H^{k_t}  (0,T) \otimes H^{2}(\Omega ) } + h_s^{k_s-2}\| u \|_{ H^{1}  (0,T) \otimes H^{k_s}(\Omega ) }\right),  
  \end{equation}
  where $k_t := \min\{q_t,p_t+1\}$, $k_s := \min\{q_s,p_s+1\}$, and $C$ is a constant that depends only on $p_t, p_s , \theta$ and the parametrization $\vect{G}$.
\end{thm}
\begin{proof}
  The result follows from the anisotropic error estimates developed in \cite{Da2012}.
  We report here only the main steps, since the proof is similar to the one of \cite[Proposition 4]{Montardini2018space}.
  The generalization of \cite[Theorem 5.1]{Da2012} to the $d+1$ dimensional case, gives the existence of a projection 
  $\Pi_h:\mathcal{V} \cap \left( H^{q_t} (0,T) \otimes H^{2}(\Omega) \right) \cap \left( H^{1} (0,T) \otimes H^{q_s}(\Omega) \right) \to \mathcal{V}_h$, 
  such that
  \begin{equation}\label{eq:stime_di_approssimazione}
    \begin{split}
    \| u - \Pi_hu\|_{ L^{2}  (0,T) \otimes L^{2}(\Omega ) }&\leq C_1  \left(h_t^{k_t-1} \| u \|_{ H^{k_t-1}  (0,T) \otimes L^{2}(\Omega ) } + h_s^{k_s-2} \| u \|_{ L^{2}  (0,T) \otimes H^{k_s-2}(\Omega ) } \right), \\
    \| u - \Pi_hu\|_{ H^{1}  (0,T) \otimes L^{2}(\Omega ) }&\leq C_2  \left(h_t^{k_t-1} \| u \|_{ H^{k_t}  (0,T) \otimes L^{2}(\Omega ) } + h_s^{k_s-2} \| u \|_{ H^{1}  (0,T) \otimes H^{k_s-2}(\Omega ) } \right), \\
    \| u - \Pi_hu\|_{ L^{2}  (0,T) \otimes H^{2}(\Omega ) }&\leq C_3  \left(h_t^{k_t-1} \| u \|_{ H^{k_t-1}  (0,T) \otimes H^{2}(\Omega ) } + h_s^{k_s-2} \| u \|_{ L^{2}  (0,T) \otimes H^{k_s}(\Omega ) } \right) .
    \end{split}
  \end{equation}
  From the following inequality 
  \begin{equation*}\strain{
    \begin{split}
      \|u-v_h\|_{\mathcal{V}}^2 &= \|u-v_h\|^2_{L^2(\mathcal{Q})} + \|\mathbb{S}(u-v_h)\|^2_{L^2(\mathcal{Q})}     \\
      &\leq \|u-v_h\|^2_{L^2(\mathcal{Q})} + 2 \|\partial_t(u-v_h)\|^2_{L^2(\mathcal{Q})}  + 2\nu \|\Delta (u-v_h)\|^2_{L^2(\mathcal{Q})}\\
      &\leq \| u - v_h\|_{ L^{2}  (0,T) \otimes L^{2}(\Omega ) }^2 
          + 2\| u - v_h\|_{ H^{1}  (0,T) \otimes L^{2}(\Omega ) }^2 
          + 2\nu\| u - v_h\|_{ L^{2}  (0,T) \otimes H^{2}(\Omega ) }^2,
    \end{split}}
  \end{equation*}  
  with the choice $v_h = \Pi_hu$, and by \eqref{eq:stime_di_approssimazione} with obvious upper bounds on the right hand side, it holds 
  $$
    \|u-\Pi_hu\|_{\mathcal{V}} \leq C \left( h_t^{k_t-1}\| u \|_{ H^{k_t}  (0,T) \otimes H^{2}(\Omega ) } + h_s^{k_s-2}\| u \|_{ H^{1}  (0,T) \otimes H^{k_s}(\Omega ) } \right),
  $$
  therefore, \eqref{eq:a-priori-error-bound} follows from Theorem \ref{thm:quasi-optimality}. 
\end{proof}
\begin{rmk}
  \label{rem:on-the-error-bound}
     In Theorem \ref{thm:quasi-optimality}, the degrees $p_t$,  $p_s$
  and the mesh-sizes $h_t$,  $h_s$ play a similar role. This motivates
  our   choice $p_t=p_s=:p$  and $h_t= h_s=:h$ for the numerical
  tests in Section \ref{sec:numerical-tests}. In this case, and if the
  solution $u$ is smooth, \eqref{eq:a-priori-error-bound} yields
  $h$-convergence of order $p-1$.  
\end{rmk}
  
\subsection{Discrete system} 
The least-squares space-time discretization \eqref{eq:ls-discrete-system} can be written as:
\begin{equation*}
   \text{Find }   u_h\in \mathcal{V}_h \text{ such that } \mathcal{A}(u_h,\mathbb{S} v_h) = \mathcal{F}(\mathbb{S} v_h) \quad \, \forall v_h\in  \mathcal{V}_h,
 \end{equation*} 
and in particular, for all $v_h \in \mathcal{V}_h$, we point out that 
\begin{equation*}\strain{}
  \begin{split}
    \mathcal{A}(u_h,\mathbb{S} v_h) &= \int_{\Omega}\int_{0}^T  \left(\mathbb{S} u_h \right) \overline{\left(\mathbb{S} v_h \right)} \,\dt \, \d\Omega \\
                                    &= \int_{\Omega}\int_{0}^T  \partial_t u_h \overline{\partial_t v_h} + \nu^2\Delta u_h \overline{\Delta v_h} + i\nu\partial_t\nabla u_h \cdot \overline{\nabla v_h} - i\nu\nabla u_h \cdot \overline{\partial_t\nabla v_h}\,\dt \, \d\Omega ,      
  \end{split}
\end{equation*}
and 
\[
\mathcal{F}(\mathbb{S} v_h )  = \int_{\Omega}\int_{0}^T f\,   \overline{\left(\mathbb{S} v_h \right)} \,\dt \, \d\Omega = \int_{\Omega}\int_{0}^T f\,   \overline{\left(i\partial_t v_h - \Delta v_h\right)} \,\dt \, \d\Omega.
\] 
Therefore, the linear system associated to \eqref{eq:ls-discrete-system} is
\begin{equation}
\mathbf{A}\mathbf{u} = \mathbf{f},
\label{eq:sys_solve}
\end{equation}
where  
$[\mathbf{A}]_{i,j}=\mathcal{A}\left( B_{j, \vect{p}}, \mathbb{S} (B_{i,  \vect{p}})\right)$ 
and $[\mathbf{f}]_{i}=\mathcal{F}\left( \mathbb{S} (B_{i,  \vect{p}})\right)$. 
The tensor-product structure of the isogeometric space \eqref{eq:disc_space} allows to write  
the system  matrix $\mathbf{A}$ as sum of Kronecker products of matrices as 
\begin{equation} 
\mathbf{A}  \   = \mathbf{M}_s  \otimes \mathbf{L}_t + \nu^2  \mathbf{B}_s \otimes \mathbf{M}_t +  \nu \mathbf{L}_s \otimes \left( \mathbf{W}_t + \mathbf{W}_t^* \right), \label{eq:syst_mat} 
\end{equation}
where   for $  i,j=1,\dots,N_s $
\begin{subequations}
\begin{equation}
\label{eq:space_mat}
  \begin{split}
    [ \mathbf{L}_s]_{i,j}  =  \int_{\Omega} \nabla  B_{i, p_s}(\vect{x})\cdot \nabla  &B_{j, p_s}(\vect{x}) \ \d\Omega,  \quad [ \mathbf{M}_s]_{i,j}  =  \int_{\Omega}  B_{i, p_s}(\vect{x}) \  B_{j, p_s}(\vect{x}) \ \d\Omega,\\
    [ \mathbf{B}_s]_{i,j}  &=  \int_{\Omega} \Delta  B_{i, p_s}(\vect{x}) \Delta  B_{j, p_s}(\vect{x}) \ \d\Omega.  
  \end{split}
\end{equation}
 
  while for $i,j=1,\dots,n_t$
  \begin{equation}
  \label{eq:time_mat}
    \begin{split}
      [ \mathbf{L}_t]_{i,j}  = \int_{0}^T   b'_{j,  {p}_t}(t)\,  &b'_{i,  {p}_t}(t) \, \dt, 
      \quad    
      [ \mathbf{M}_t]_{i,j} = \int_{0}^T\,  b_{i, p_t}(t)\,  b_{j, p_t}(t)  \, \dt , \\
      [ \mathbf{W}_t]_{i,j}  &= \mathrm{i}\int_{0}^T   b'_{j,  {p}_t}(t)\,  b_{i,  {p}_t}(t) \, \dt, 
    \end{split}
  \end{equation}

 \label{eq:pencils}
\end{subequations}

\subsection{Ultraweak space-time method}
Here we recall also the following ultraweak discretization that has been proposed in \cite{hain2022ultra}.
Let  $\mathcal{W}_h:= \mathcal{X}_{h,T}$ be the isogeometric space
with final conditions endowed with the $\|\cdot\|_{\mathcal{W}}$-norm, 
and fix $\mathcal{V}_h:= \mathbb{S}(\mathcal{W}_h)$ endowed with the $\|\cdot\|_{L^2(\mathcal{Q})}$-norm.
Notice that, $\forall v_h \in \mathcal{V}_h, \, w_h \in \mathcal{W}_h$, it holds 
\begin{equation*}
 \mathcal{A}(v_h,w_h) = \left(\mathbb{S}(v_h),w_h\right) = \left(v_h,\mathbb{S}(w_h)\right) - \mathrm{i}\left( v_h (\cdot,0) , {w}_h(\cdot,0) \right)_{L^2(\Omega)},
\end{equation*}
with $\left(\cdot , \cdot \right)_{L^2(\Omega)}$ denoting the complex scalar product in $L^2(\Omega)$. Therefore, introducing the sesquilinear form
\begin{equation}
 \label{eq:uw_sesquilinear_form}
 \mathcal{A}_{\mathtt{uw}}(v_h, w_h) := \left(v_h,\mathbb{S}(w_h)\right), \quad \forall v_h \in \mathcal{V}_h,\, w_h \in \mathcal{W}_h,
\end{equation}
we have the following ultraweak method for \eqref{eq:var_for}: 
\begin{equation}
\label{eq:uw-discrete-system}
 \text{Find }   u_h\in \mathcal{V}_h \text{ such that } 
   \mathcal{A}_{\mathtt{uw}}(u_h, w_h) = \mathcal{F}(w_h)  + \mathrm{i}\left( u_0 , {w}_h(\cdot,0) \right)_{L^2(\Omega)}
     \quad \, \forall w_h\in  \mathcal{W}_h,
\end{equation} 
where now the right hand side contains eventually the initial data $u_0$. 
As regards the well posedness and stability of \eqref{eq:uw-discrete-system} we refer to \cite{hain2022ultra}. 
\begin{rmk}
  The discrete system associated to \eqref{eq:uw-discrete-system} has the same Kronecker structure as \eqref{eq:syst_mat}, with appropriate final conditions at $T$.
\end{rmk}

\section{Fast solver for the parametric domain}
In this section we focus on the case $\vect{F} = {Id}$, that is $\mathcal{Q} = (0,T) \times (0,1)^d$ 
is the parametric domain in space and a finite interval in time direction. In this context we are able to introduce a stable and fast solver for problem \eqref{eq:ls-discrete-system}.
We introduce, for the system \eqref{eq:sys_solve}, the  preconditioner  
\begin{equation}\label{eq:prec_definition}
 \widehat{\mathbf{P}}:=  \widehat{\mathbf{M}}_s \otimes \mathbf{L}_t  + \nu^2   \widehat{\mathbf{L}}_s^{T}\widehat{\mathbf{M}}_s^{-1}\widehat{\mathbf{L}}_s  \otimes \mathbf{M}_t +\nu \widehat{\mathbf{L}}_s \otimes (\mathbf{W}_t + \mathbf{W}_t^*),
\end{equation}
where the matrices $\mathbf{L}_t,\mathbf{M}_t$ and $\mathbf{W}_t$ are defined in \eqref{eq:time_mat}, while $\widehat{\mathbf{L}}_s$ and $\widehat{\mathbf{M}}_s$ are   
\begin{equation*}
	\widehat{\mathbf{L}}_s = \sum_{l = 1}^{d} \widehat{\mathbf{M}}_d \otimes \dots \otimes \widehat{\mathbf{M}}_{l+1} \otimes \widehat{\mathbf{L}}_l \otimes \widehat{\mathbf{M}}_{l-1} \otimes \dots \otimes \widehat{\mathbf{M}}_1,
	\quad \text{and} \quad
	\widehat{\mathbf{M}}_s = \widehat{\mathbf{M}}_d \otimes \dots \otimes \widehat{\mathbf{M}}_1,
\end{equation*}
and for $l=1,\dots,d$, with indexes $i,j=1,\dots,n_{l}$, it holds 
$$
[\widehat{\mathbf{L}}_l]_{i,j} := \int\limits_{0}^{1}  \widehat{b}'_{j,p}(x_l) \widehat{b}'_{i,p}(x_l) \mathrm{d}x_l, 
\quad \text{ and } \quad 
[\widehat{\mathbf{M}}_l]_{i,j} := \int\limits_{0}^{1} \widehat{b}_{j,p}(x_l) \widehat{b}_{i,p}(x_l) \mathrm{d}x_l. 
$$

The efficient application of the proposed preconditioner, that is, the
solution of a linear system with matrix $ \widehat{\mathbf{P}}$, should
exploit  the structure highlighted above.   When the
pencils $ ({ \widehat{\mathbf{L}}_1},\widehat{\mathbf{M}}_1 ), \ldots,  ( { \widehat{\mathbf{L}}_d},\widehat{\mathbf{M}}_d  )$ 
admit a stable generalized eigendecomposition,   a possible approach is 
the Fast Diagonalization (FD) method, see \cite{Deville2002} and \cite{Lynch1964}  for details.
 
\subsection{Stable factorization of \mathinhead{( \widehat{\mathbf{L}}_l, \widehat{\mathbf{M}}_l)}{(Ll,Ml)} for \mathinhead{l=1,\dots,d}{l=1,...,d} }
\label{sec:stable_space}

The spatial stiffness and mass matrices ${  \widehat{\mathbf{L}}_l}$ and $\widehat{\mathbf{M}}_l$ are symmetric and positive definite for  $l=1,\dots,d$. Thus, the pencils 
  $ ({ \widehat{\mathbf{L}}_l}, \widehat{\mathbf{M}}_l )$ for  $l=1,\dots,d$ admit 
    the  generalized eigendecomposition  \[
      {  \widehat{\mathbf{L}}_l}\mathbf{U}_l =
\widehat{\mathbf{M}}_l\mathbf{U}_l\mathbf{\Lambda}_l,
\] where the matrices $
\mathbf{U}_l$ contain  in each column the
$\widehat{\mathbf{M}}_l$-orthonormal  generalized eigenvectors and 
  $\mathbf{\Lambda}_l$ are diagonal matrices whose entries contain the generalized eigenvalues.
Therefore we have for $l=1,\dots ,d$ the  factorizations 
\begin{equation}
\label{eq:space_eig}
\mathbf{U}^T_l {  \widehat{\mathbf{L}}_l} \mathbf{U}_l=
\mathbf{\Lambda}_l  \quad\text{ and  }\quad \mathbf{U}^T_l \widehat{\mathbf{M}}_l \mathbf{U}_l=
\mathbb{I}_{  n_{s,l}},
\end{equation}
 where $\mathbb{I}_{n_{s,l}}$ denotes the identity matrix of dimension $n_{s,l} \times
n_{s,l}$. 
{  Figure \ref{fig:autovettori_laplaciano} shows the shape of the generalized eigenvectors in $\mathbf{U}_l$, with associated eigenvalue in $\Lambda_l$, for a fixed univariate direction $l=1,\dots,d$ discretized with degree $p_s=3$ B-Splines and uniform partition.}
The stability of the decomposition is expressed by the condition number of the eigenvector matrix.
In particular  $\mathbf{U}^T_l \widehat{\mathbf{M}}_l \mathbf{U}_l=
\mathbb{I}_{  n_{s,l}}$ implies  that
 \[ \kappa_2 (\mathbf{U}_l) := 
      \|\mathbf{U}_l\|_{2 } \|\mathbf{U}_l^{-1}\|_{2 }  =
           \sqrt{\kappa_2(\widehat{\mathbf{M}}_l)},
\]  
where  $   \|\cdot\|_{ 2  }$ is the norm induced by the Euclidean vector norm. 
The condition number $\kappa_2(\widehat{\mathbf{M}}_l)$ has been
studied theoretically in \cite{gahalaut2014condition} and numerically in \cite{Montardini2018space} 
and it does not depend on the mesh-size,  but it depends on the polynomial degree.
Indeed,  we report  in Table \ref{tab:cond_number_space} the behavior of 
$\kappa_2 (\mathbf{U}_l)$   for different values of  spline degree $p_s$ 
and for different uniform discretizations  with  number of elements denoted by $n_{el}$.     
We observe  that $\kappa_2 (\mathbf{U}_l)$  exhibits a   dependence only on  
$p_s$, but stays moderately low for all low polynomial degrees that
are in the range of interest. 
\begin{figure}
  \centering
  \includegraphics[width = \textwidth]{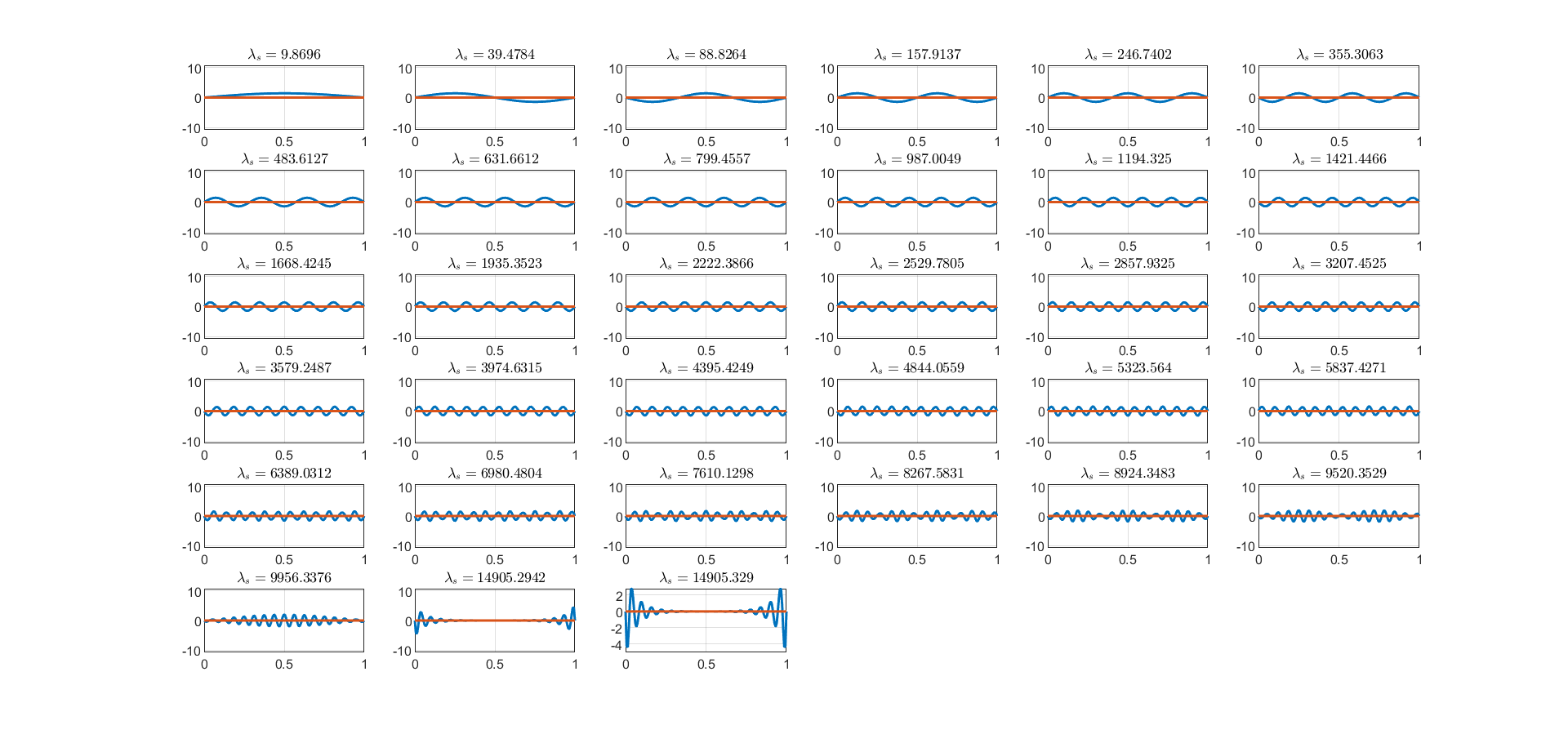}
  \caption{Generalized eigenvectors in space with associated eigenvalues for $p_s = 3$ and $n_{el}= 32$ elements. The real part is expressed in blue, while the imaginary part (null) is in red.}
  \label{fig:autovettori_laplaciano}
\end{figure}

{\renewcommand\arraystretch{1.2} 
\begin{table}[H]                                                        
	\centering                                                              
	\begin{tabular}{|c|c|c|c|c|c|c|c|}                                      
		\hline                                                                  
		 $n_{el}$ & $p_s=2$ & $p_s=3$ & $p_s=4$ & $p_s=5$ & $p_s=6$ & $p_s=7$ & $p_s=8$ \\                                              
		\hline                                                                  
		\z\z32 & $2.7 \cdot 10^0 $ & $ 4.5 \cdot 10^0$ & $ 7.6 \cdot 10^0$ & $ 1.3 \cdot 10^1$ & $ 2.1  \cdot 10^1$ & $ 3.5  \cdot 10^1$ & $ 5.7  \cdot 10^1$ \\   
		\hline                                                          
		\z\z64  & $ 2.7 \cdot 10^0 $ & $ 4.5 \cdot 10^0$ & $ 7.6 \cdot 10^0$ & $ 1.3  \cdot 10^1$ & $ 2.1  \cdot 10^1$ & $ 3.5 \cdot 10^1$ & $ 5.7  \cdot 10^1$ \\   
		\hline                                                          
		\z128 & $ 2.7 \cdot 10^0$ & $ 4.5 \cdot 10^0$ & $ 7.6 \cdot 10^0$ & $ 1.3 \cdot 10^1 $ & $ 2.1  \cdot 10^1$ & $ 3.5  \cdot 10^1$ & $ 5.7 \cdot 10^1$ \\   
		\hline                                                          
		\z256  & $ 2.7 \cdot 10^0$ & $ 4.5 \cdot 10^0$ & $ 7.6 \cdot 10^0$ & $ 1.3  \cdot 10^1$ & $ 2.1  \cdot 10^1$ & $ 3.5  \cdot 10^1$ & $ 5.7  \cdot 10^1$ \\  
		\hline                                                          
		\z512  & $ 2.7 \cdot 10^0$ & $ 4.5 \cdot 10^0$ & $ 7.6 \cdot 10^0$ & $ 1.3  \cdot 10^1$ & $ 2.1  \cdot 10^1$ & $ 3.5  \cdot 10^1$ & $ 5.7 \cdot 10^1$ \\   
		\hline                                                          
		1024  & $ 2.7 \cdot 10^0$ & $ 4.5 \cdot 10^0$ & $ 7.6 \cdot 10^0$ & $ 1.3  \cdot 10^1$ & $ 2.1  \cdot 10^1$ & $ 3.5  \cdot 10^1$ & $ 5.7 \cdot 10^1 $ \\ 
		\hline                                                                  
	\end{tabular}                                                           
\caption{$\kappa_2 (\mathbf{U}_l)$ for different polynomial degrees $p_s$ and number of elements $n_{el}$.}                                                                               
\label{tab:cond_number_space}               
\end{table}}

Moreover, in \cite{henning2022ultraweak} it is shown that there is spectral equivalence between $\widehat{\mathbf{B}}_s$ and $\widehat{\mathbf{L}}_s^T \widehat{\mathbf{M}}_s^{-1}\widehat{\mathbf{L}}_s $. 
We investigate numerically this spectral equivalence, and Figure \ref{fig:equivalenza_in_spazio} 
shows the eigenvalues of $(\widehat{\mathbf{L}}_s^T \widehat{\mathbf{M}}_s^{-1}\widehat{\mathbf{L}}_s )^{-1} \widehat{\mathbf{B}}_s $ are 
{  clustered and close to 1, for splines of degree $p=2,3,4$ and different uniform partitionas with $n_{el} = 8, 16, 32, 64, 128$}. In conclusion the spectral equivalence is stable under mesh refinement.

As regards the time pencils, the spectral equivalence between $\widehat{\mathbf{L}}_t$ and $\widehat{\mathbf{W}}_t^*\widehat{\mathbf{M}}_t^{-1}\widehat{\mathbf{W}}_t $
is unstable under mesh refinement, see Figure \ref{fig:equivalenza_in_tempo} where we performed the analogous test, therefore 
we kept the full structure of the time pencils in the preconditioner. 

\begin{figure}
  \centering
  \subfloat[]
  [Eigenvalues of $(\widehat{\mathbf{L}}_s^T \widehat{\mathbf{M}}_s^{-1} \widehat{\mathbf{L}}_s)^{-1}\widehat{\mathbf{B}}_s$ for different degrees $p$ and number of elements $n_{el}$. \label{fig:equivalenza_in_spazio}]
  {\includegraphics[width = 0.7\textwidth]{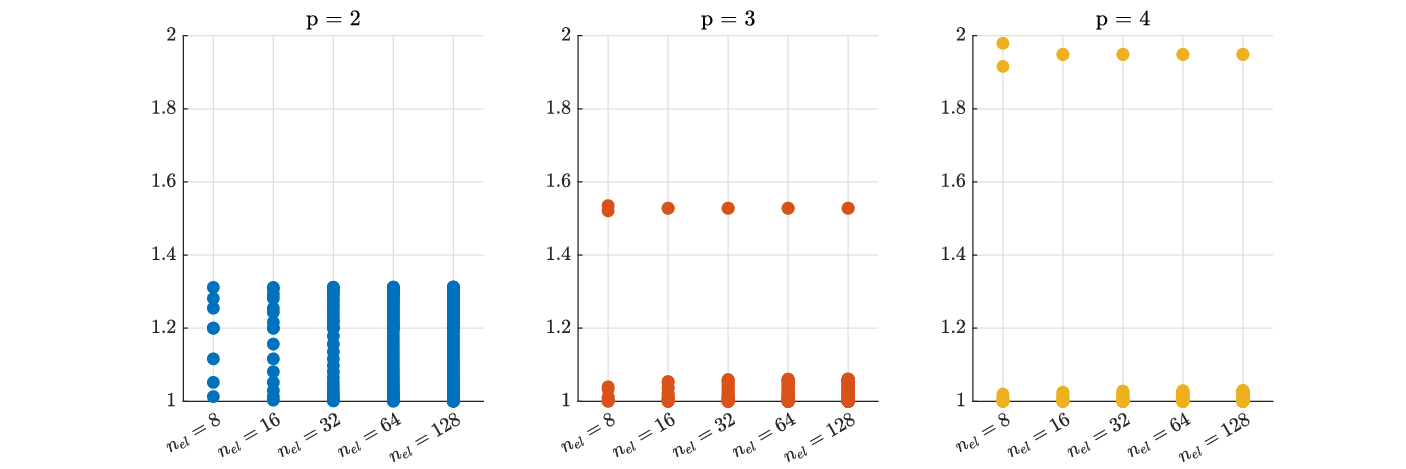}}\\
  \subfloat[]
  [Eigenvalues of $(\widehat{\mathbf{W}}_t^* \widehat{\mathbf{M}}_t^{-1} \widehat{\mathbf{W}}_t)^{-1}\widehat{\mathbf{L}}_t$ for different degrees $p$ and number of elements $n_{el}$. \label{fig:equivalenza_in_tempo}]
  {\includegraphics[width = 0.7\textwidth]{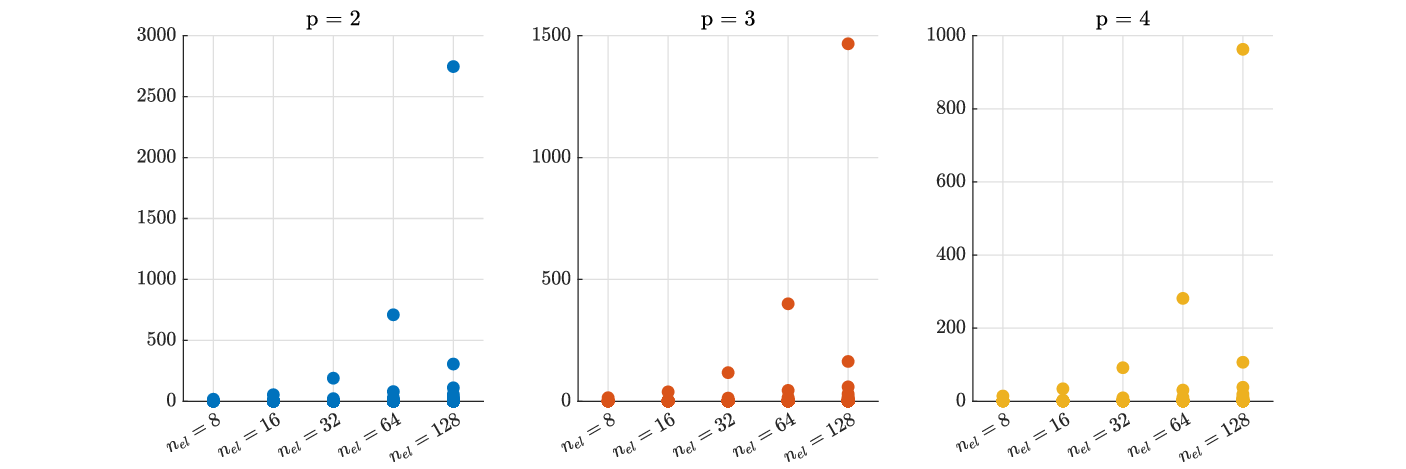}}
  \caption{Numerical investigation of spectral equivalence.}
  \label{fig:spectral_equivalence_in_preconditione_blocks}
\end{figure}


\subsection{Application of the preconditioner}
\label{sec:application_of_prec}
The application of the preconditioner involves the solution of the linear system 
\begin{equation}
\label{eq:prec_appl_2}
 \widehat{\mathbf{P}}\mathbf{s}=\mathbf{r},
\end{equation}
where $\widehat{\mathbf{P}}$ has the structure
\eqref{eq:prec_definition}.
We are able to efficiently solve system \eqref{eq:prec_appl_2} by the Fast Diagonalization method. 
The starting point, is the setup of the preconditioner, 
that is the factorizations \eqref{eq:space_eig} of  the  pencils
$({  \widehat{\mathbf{L}}_l}, \widehat{\mathbf{M}}_l)$ for  $l=1,\dots,d$.

Then, define $ \mathbf{U}_s:=\mathbf{U}_d\otimes\dots\otimes\mathbf{U}_1$ 
and $\mathbf{\Lambda}_s:=\sum_{l=1}^d \mathbb{I}_{n_{s,d}}\otimes\dots\otimes\mathbb{I}_{n_{s,l+1}} \otimes \mathbf{\Lambda}_l \otimes \mathbb{I}_{n_{s,l-1}}\otimes\dots\otimes\mathbb{I}_{n_{s,1}}$.
Notice that  $\widehat{\mathbf{M}}_s^{-1} = {\mathbf{U}}_s{\mathbf{U}}_s^T$, therefore 
the matrix $\widehat{\mathbf{L}}_s \widehat{\mathbf{M}}_s^{-1} \widehat{\mathbf{L}}_s $ admits the stable 
factorization
\begin{equation*}
  \label{eq:factorization_LML}
  \mathbf{U}^T_s \widehat{\mathbf{L}}_s \widehat{\mathbf{M}}_s^{-1} \widehat{\mathbf{L}}_s  \mathbf{U}_s = 
  \mathbf{\Lambda}_s^2.
\end{equation*}

The preconditioner $\widehat{\mathbf{P}}$ admits the following factorization
\begin{equation}\label{eq:preconditioner_str2}
  \widehat{\mathbf{P}} = 
    \left( \mathbf{U}_s^T \kron \mathbb{I}_{n_t}\right)^{-1} 
    \left(        \mathbb{I}_{N_s}       \kron \mathbf{L}_t 
          + \nu^2 \mathbf{\Lambda}_s^2   \kron \mathbf{M}_t 
          + \nu   \mathbf{\Lambda}_s     \kron \left(\mathbf{W}_t+\mathbf{W}^*_t\right)\right)
    \left(\mathbf{U}_s \kron \mathbb{I}_{n_t}\right)^{-1}.
\end{equation}
Note that the second factor in \eqref{eq:preconditioner_str2} that is 
\[ 
  \mathbf{H} := 
  \left(         \mathbb{I}_{N_s}       \kron \mathbf{L}_t 
        + \nu^2  \mathbf{\Lambda}_s^2   \kron \mathbf{M}_t 
        + \nu    \mathbf{\Lambda}_s     \kron \left(\mathbf{W}_t+\mathbf{W}^*_t\right)
  \right)
\]
is sum of three Kronecker matrices, whose space factors are diagonal matrices. We have the following block diagonal structure 
\begin{equation*}
    \mathbf{H} = 
    \begin{bmatrix}
        \mathbf{H}_1  &        &  \\
                      & \ddots &  \\
                      &        & \mathbf{H}_{N_s}
    \end{bmatrix},
\end{equation*}
where  $\mathbf{H}_i$, for $ i=1,\dots,N_s$, are banded matrices with bandwidth $2p_t+1$ defined as 
\[
\mathbf{H}_i:= \mathbf{L}_t +{  \nu^2} [\mathbf{\Lambda}_s^2 ]_{i,i} \mathbf{M}_{n_t} + \nu [\mathbf{\Lambda}_s ]_{i,i} \left(\mathbf{W}_t+\mathbf{W}^*_t\right).
\]
In order to invert $\mathbf{H}$, it is now sufficient to invert the following independent $N_s$ problems of size $n_t \times n_t$:   
\begin{equation}\label{eq:independent_problems}
  \mathbf{H}_i  \mathbf{x}_{i} = \mathbf{y}_i \quad \text{ for } i = 1,\dots, N_s.  
\end{equation}

Summarizing, the solution of \eqref{eq:prec_appl_2} can be computed  by the following algorithm.
\begin{algorithm}[H]
  \caption{Fast Diagonalization}\label{al:direct_P_ls}
  \begin{algorithmic}[1]
  \State Compute the factorizations \eqref{eq:space_eig}.   
  \State Compute  $ \mathbf{y} = (\mathbf{U}_s^T \otimes \mathbb{I}_t )\mathbf{r} $. 
  \State Compute  $ \mathbf{x}_i = \mathbf{H}_i^{-1} \mathbf{y}_i \quad \text{ for } i = 1,\dots, N_s. $
  \State Compute  $ \mathbf{s} = (\mathbf{U}_s \otimes \mathbb{I}_t)\ \widetilde{\mathbf{s}}. $
  \end{algorithmic}
\end{algorithm}

We conclude with the following remark for a possible parallel implementation of Algorithm \ref{al:direct_P_ls}.
\begin{rmk}\label{rm:time_and_space}
  The decision to consider time as the first variable allows us to write the matrix $\mathbf{H}$ in a block diagonal form. 
  In view of an efficient parallel implementation, this natural diagonal block structure does not require data shuffling, 
  decreasing the communication cost between nodes.
\end{rmk}

\subsection{Computational cost and memory requirements}
In this section we discuss the computational costs and memory requirements in the implementation of Algorithm \ref{al:direct_P_ls}.
First, notice that the matrix $\mathbf{A}$ in \eqref{eq:syst_mat}  is symmetric positive definite therefore we choose Conjugate Gradients (CG) as linear solver for solving  the system \eqref{eq:sys_solve}.
Clearly,  the computational cost of each iteration of the CG solver depends on both the preconditioner setup and  
application cost. 

We assume for simplicity that, for each univariate direction $l=1,\dots,d$, the space matrices have dimension $n_s \times n_s$, while the time 
matrices involved in the preconditioners, have dimension $n_t \times n_t$.  
Thus the total number of degrees-of-freedom is $N_{dof} = N_sn_t = n_s^dn_t$.  

The setup  of $\widehat{\mathbf{P}}$ includes the operations performed in Step 1 of Algorithm \ref{al:direct_P_ls}, 
i.e. $d$ spatial eigendecompositions, that have a total cost of $O(dn_s^3)$ FLOPs, and the construction of the block diagonal matrix $\mathbf{H}$, which costs $O(p_tn_tN_s) = O(p_tN_{dof})$. 
We remark that the setup of the preconditioners has to be performed only once, 
since the matrices involved do not change during the iterative procedure. 

The application of the preconditioner is performed by Steps 2-4 of Algorithm \ref{al:direct_P_ls}.
Exploiting \eqref{eq:kron_vec_multi},  Step 2 and Step 4 costs $O(dn_s^{d+1}n_t) = O(dn_sN_{dof})$ FLOPs.  
The cost of solving each sparse problem \eqref{eq:independent_problems} makes the cost for Step 3 equal to  $O(p_t^2 n_t N_{s}) = O(p_t^2 N_{dof})$ FLOPs.  

In conclusion, the total cost of Algorithm \ref{al:direct_P_ls} is $ O(d n_s^3) + O(d n_s N_{dof}) + O(p_t^2 N_{dof})$ FLOPs. 
The non-optimal dominant cost of Step 2 and Step 4 is determined by the dense matrix-matrix products.
However, these operations are usually implemented on modern computers in a very efficient way and  
the overall serial computational time grows almost as $O(N_{dof})$, see i.e. \cite{Montardini2018space,loli2020efficient}

The other dominant computational cost in a CG iteration is the cost of the residual computation. 
In Algorithm \ref{al:direct_P_ls}, this involves the multiplication of the matrix $\mathbf{A}$  with a vector. 
This multiplication is done by exploiting the special structure \eqref{eq:syst_mat}, that allows a matrix-free approach 
and the use of formula \eqref{eq:kron_vec_multi}. With the matrix-free approach, noting that the time matrices $\mathbf{L}_t, \mathbf{M}_t, \mathbf{W}_t$ 
are banded matrices with bandwidth $2p_t + 1$, and the spatial matrices $\mathbf{J}_s,\mathbf{L}_s,\mathbf{M}_s$ 
have a number of non-zeros per row equal to $(2p_s+1)^d$, the computational cost of a single matrix-vector product is 
$O(N_{dof} p^d)$ FLOPs, if we assume $p = p_s \approx p_t$. The dominant cost in the iterative solver is represented by the residual computation. 
This is a typical behaviour of the FD-based preconditioning strategies, see \cite{Montardini2018space,Sangalli2016,Montardini2018}.

We now investigate the memory consumption of the preconditioning strategy proposed.
For the preconditioner, we have to store the eigenvector spatial matrices, $\mathbf{U}_1, \dots, \mathbf{U}_d$, 
the diagonal matrices $\mathbf{\Lambda}_1, \dots, \mathbf{\Lambda}_d$ and the banded time pencils $\mathbf{L}_t,\mathbf{M}_t$ and $\mathbf{W}_t$ 
of size $n_t \times n_t$. The memory required is roughly 
\[ 
  O(d (n_s^2+n_s)) + O(p_t N_{dof}) + O(p_t n_t).
\]

For the system matrix $\mathbf{A}$, in addition to the time factors $\mathbf{L}_t,\mathbf{M}_t$ and $\mathbf{W}_t$, 
we need to store the spatial factors $\mathbf{M}_s,\mathbf{B}_s$ and $\mathbf{L}_s$. Thus the memory further required is roughly
\[
  O(p_s^d N_s) .
\] 

{ These numbers show that memory-wise our space-time strategy is very appealing when compared to other
approaches, even when space and time variables are discretized separately, e.g., with finite differences in
time or other time-stepping schemes. For example if we assume $d=3$, $p_t\approx p_s=p$ and $n_t^2\leq Cp^3N_s$, 
then the total memory consumption is $O(p^3N_s+N_{dof})$, 
that is equal to the sum of the  memory needed to store the Galerkin matrices associated to spatial variables 
and the memory needed  to store the solution of the problem.
}

We remark that we could avoid storing the factors of $\mathbf{A}$ by using the matrix-free approach of \cite{Sangalli2018}. 
The memory and the computational cost of the iterative solver would significantly improve, both for the setup and the matrix-vector multiplications. 
However, we do not pursue this strategy, as it is beyond the scope of this paper.

\section{Numerical Results}
\label{sec:numerical-tests}
This section is devoted to the computation of the solution of Schr{\"o}dinger problem \eqref{eq:problem}, and to its extension to non-homogeneous conditions, with the discretization proposed in \eqref{eq:ls-discrete-system}.
We first present the numerical experiments that assess the convergence behavior  of the least squares Petrov-Galerkin approximation and then we analyze the performance of the preconditioners.

We consider only sequential executions and we force the use of a single computational thread in a Intel Core i5-1035G1 processor, running at 1 GHz and with 16 GB of RAM.
 
The tests are performed with Matlab R2023a and GeoPDEs toolbox \cite{Vazquez2016}. 
We use  the \texttt{eig} Matlab function to compute the  generalized eigendecompositions present in  Step 1 of Algorithm \ref{al:direct_P_ls},
while  Tensorlab toolbox \cite{Sorber2014} is employed to perform the  multiplications with Kronecker matrices occurring in Step 2 and Step 4.  The  solution of the linear systems \eqref{eq:independent_problems} in Step 3
is performed pagewise by Matlab direct solver (pagewise backslash operator \texttt{pagemldivide}).
The linear system is solved by CG, with tolerance   equal to  $10^{-8}$  and with   the null vector as initial guess in all tests. 

\begin{figure}
  \centering
  \subfloat[]
  [\textit{Left - }Real part of the smooth gaussian solution computed with space-time splines discretization of degree $p=3$, over a uniform mesh  with $64$ elements in space and $128$ elements in time. \textit{Right - }Real part of the exact solution.\label{fig:soluzione_gaussiana}]
  {\includegraphics[width=0.9\textwidth]{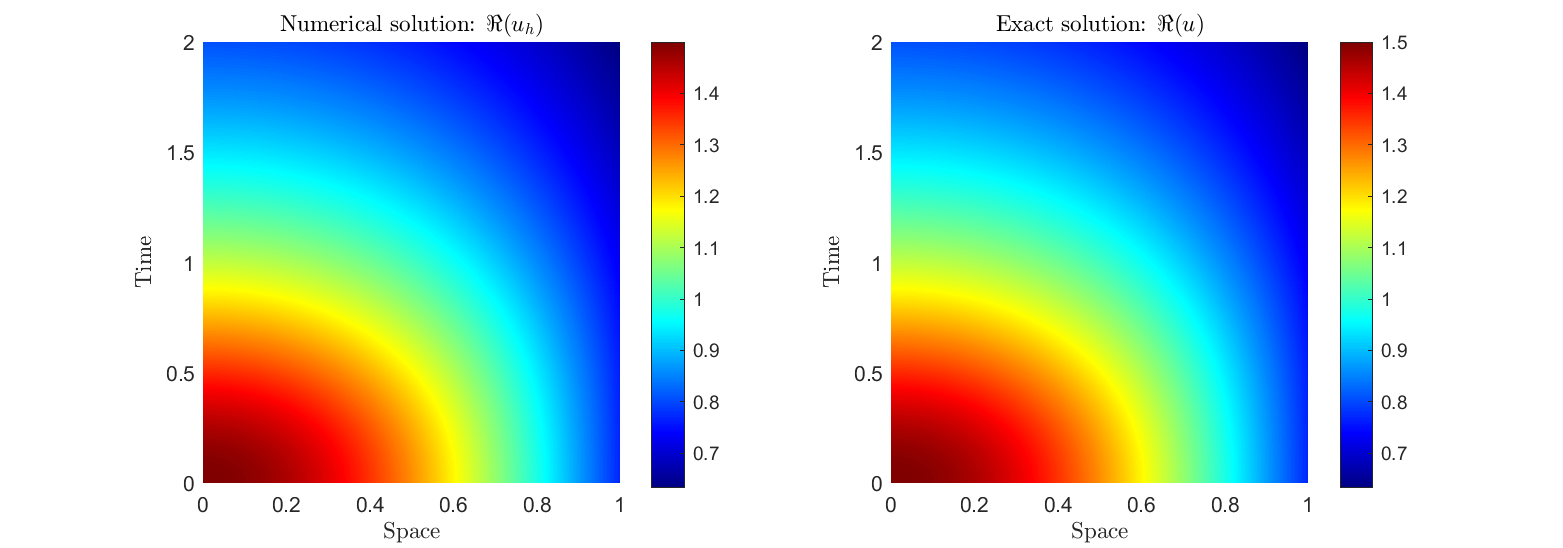} }
  \vskip 0.2cm
  \subfloat[]
  [Error convergence in ${\mathcal{V}}$-norm.\label{fig:convergenza_gaussiana}]
  {\includegraphics[width=0.45\textwidth]{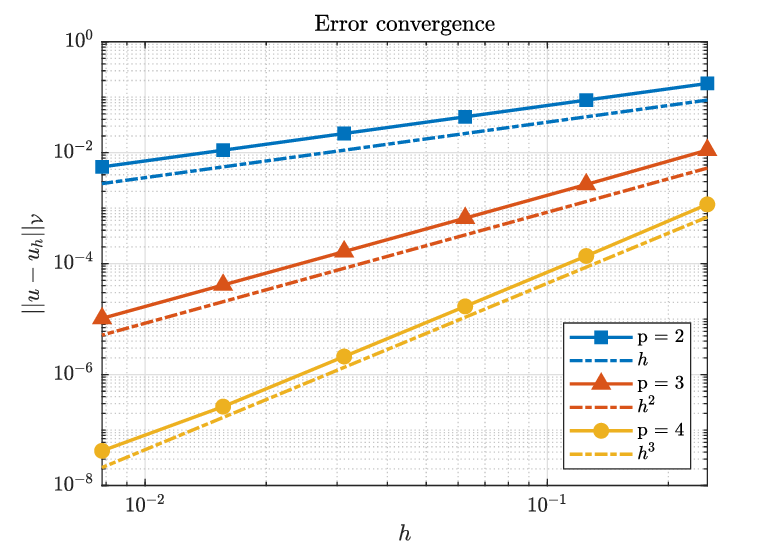} }
  \subfloat[]
  [Error convergence in ${L^2(\mathcal{Q})}$-norm.\label{fig:convergenza_gaussiana_solo_norma_l2}]
  {\includegraphics[width=0.45\textwidth]{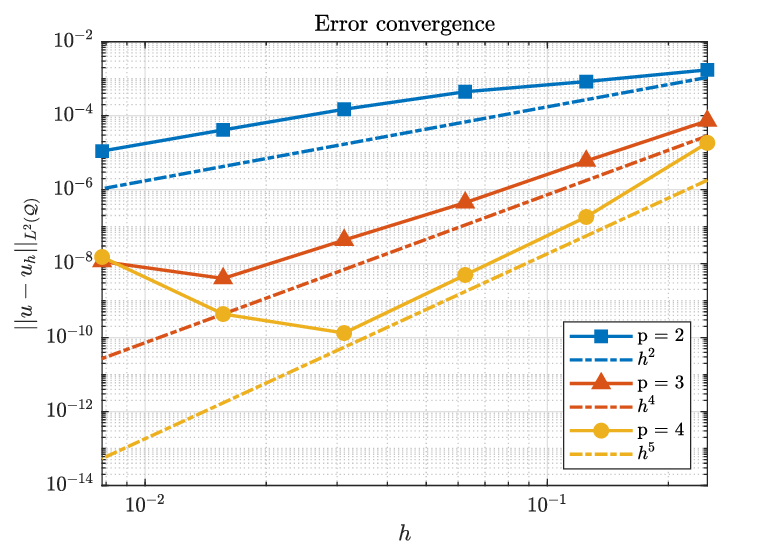} }
  \caption{Smooth solution and error convergence of the space-time least squares discretization.}
  \label{fig:soluzione_gaussiana_e_convergence}
\end{figure}

{ According to Remark \ref{rem:on-the-error-bound}, we use  the same
mesh-size in space and in time $h_s=h_t=:h$,  and use
splines of maximal continuity  and  same degree in
space and in time  $p_t=p_s=:p$.} For the sake of simplicity, we
also consider uniform knot vectors, and  denote the number of
elements in each parametric direction  by $n_{el}:=\frac{1}{h}$.

{  In out tables, the symbol $``\ast \ast"$  denotes  that the invertion of the matrix $\mathbf{A}$ in \eqref{eq:syst_mat}, 
  by Matlab direct solver backslash operator $``\backslash"$, requires more than 2 hours of computational time,
  while the symbol  $``\ast "$ indicates that the number of iterations in the CG solver exceeds the upper bound set to $200$ iterates. 
  We remark that  in all the tables the total solving time of the iterative strategies includes also the setup time of the considered preconditioner.
}  

\subsection{Orders of convergence} \label{sec:orders_conv}
Consider the Schr{\"o}dinger equation as modeled in \eqref{eq:problem}, for the space time domain 
$\mathcal{Q} = (0,T)\times (0,1)$, with $T=2$. The reference solution is the complex gaussian 
\begin{equation}
    u(t,x) = \dfrac{\alpha \beta }{\sqrt{\beta^2-\mathrm{i}\gamma t}} \exp \left\{ -\dfrac{x^2} {\beta^2-\mathrm{i}\gamma t} \right\},
\end{equation}
where $\alpha = \beta = 1.5$ and $\gamma = 2.5$. Here the Dirichlet 
boundary condition is $u|_{\Gamma_D}$, the initial condition is $u(0,x)$ and the right hand side is $f = A u$. 
The problem is discretized with a uniform mesh in both space and time directions. 
The solution for $p= 3$ is shown in Figure \ref{fig:soluzione_gaussiana}.  
In Figure \ref{fig:convergenza_gaussiana} it is shown the convergence analysis of the error under $h$-refinement
and for different polynomial degrees, for instance $p=2,\dots,4$. 
The errors are computed both with $\| \cdot \|_{\mathcal{V}}$-norm, for which the convergence Theorem \ref{thm:quasi-optimality} holds,
and with $\| \cdot \|_{L^2(\mathcal{Q})}$-norm, even if this case is not covered by theoretical results. 
For this smooth solution, the error study reveals optimal convergence under $h$-refinement in $\mathcal{V}$-norm, 
{ and the $L^2(\mathcal{Q})$ error is lower than the residual component in the $\mathcal{V}$-norm.}

\begin{figure}
  \centering
  \subfloat[]
  [\textit{Left - }Real part of the non-regular solution computed with space-time splines discretization of degree $p=4$, over a uniform mesh with $512$ elements in space and $1024$ elements in time. \textit{Right - }Real part of the exact solution.\label{fig:soluzione_demkowicz}]
  {\includegraphics[width=0.9\textwidth]{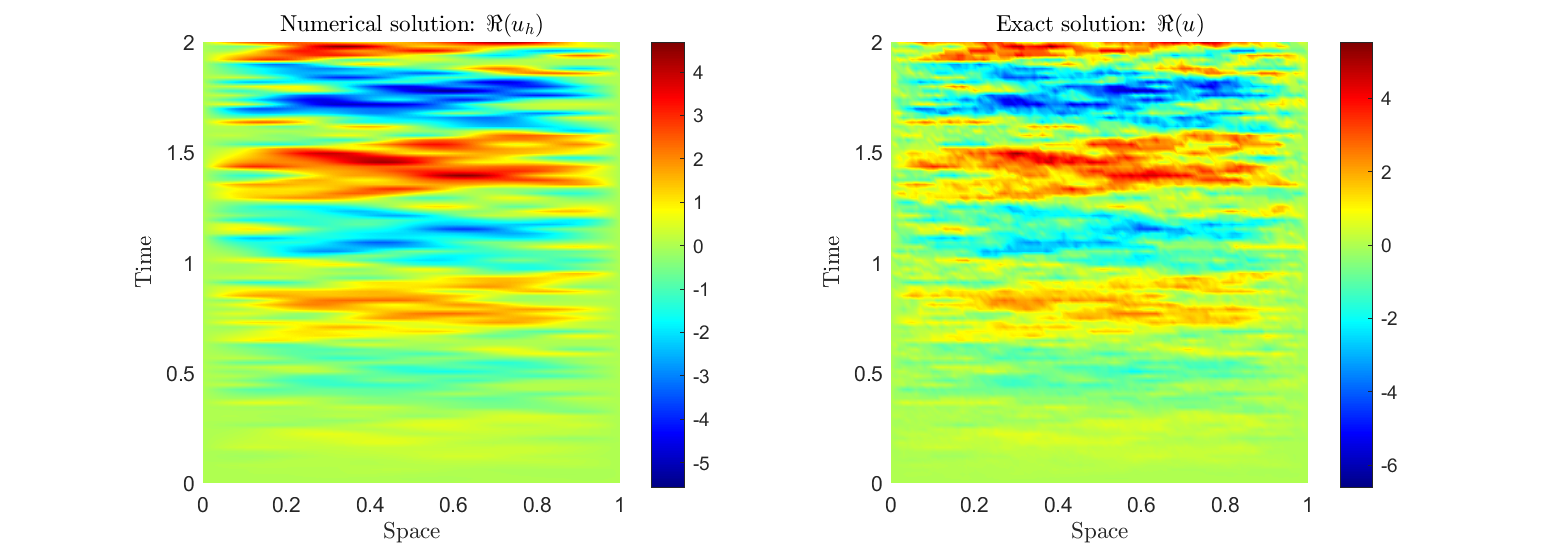} }
  \vskip 0.2cm
  \subfloat[]
  [Error convergence in ${\mathcal{V}}$-norm.\label{fig:convergenza_soluzione_demkowicz}]
  {\includegraphics[width=0.45\textwidth]{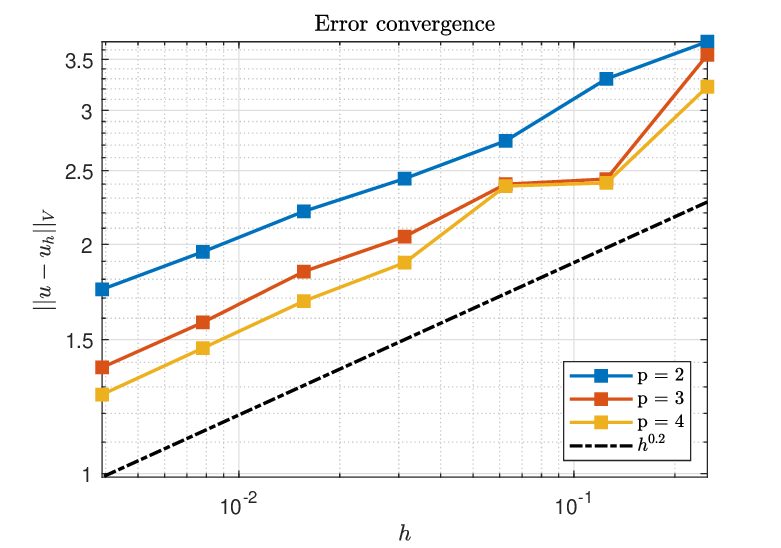} }
  \caption{Non-regular solution and error convergence of the space-time discretization.}
  \label{fig:soluzione_demkowicz_e_convergence}
\end{figure}

The second test considers the following example from \cite{DEMKOWICZ}.
Consider the space domain $\Omega = (0,1)$, the final time $T=2$, and the space-time domain $\mathcal{Q} = (0,T)\times (0,1)$. Homogeneous Dirichlet boundary 
conditions are considered on $\Gamma_D$. Let us denote by $e_k $ and $\omega^2_k$, which is, for $k = 1,2,\dots $, an eigenpair of
$$
-\Delta e_k = \omega^2_k e_k, \quad \text{a.e. in } \Omega.
$$   
By normalizing $e_k$ such that $\|e_k\|_{L^2(\Omega)} = 1$, we consider 
$
f(t,x) = \sum_{k=1}^{+\infty} f_k(t) e_k(x), 
$
where $f_k(t)$ are the Fourier coefficients of $f$ decomposed in the orthonormal 
basis $e_k$ at a given time $t$. By the following specific choice of coefficients
$$
f_k(t) = \dfrac{1}{k} \exp \left\{i\omega_k^2t\right\} \quad \text{for } k = 1,2,\dots,
$$
we considered as right hand side in \eqref{eq:problem} the following high mode truncated expansion
$$
f_M = \sum_{k=1}^M \dfrac{1}{k} \exp \left\{i\omega_k^2t\right\} e_k(x) ,
$$
with $M \gg 0$. Notice that, the solution to \eqref{eq:problem} with this specific right-hand side, is
$$
u(t,x) = \sum_{k=1}^{M} \dfrac{-it}{k} \exp \left\{ i\omega_k^2t\right\} e_k(x).
$$
We computed the solution for different polynomial degrees, on a uniform mesh, for an high mode right hand side 
$f_M$, with $M = 625$. In \ref{fig:soluzione_demkowicz} it is plotted the real part of the numerical solution for splines 
with degree $p=4$, together with the real part of the explicit solution. The solution of such a problem is non-regular and it can be shown that 
$u(t,\cdot) \in H^{1/2}(\Omega)$, while $u(\cdot,x) \in H^{1/4}(0,T)$. 
Figure \ref{fig:convergenza_soluzione_demkowicz} shows the error convergence for the high mode right hand side $f_M$, with $M = 652$ modes, 
that is optimal for each polynomial degree $p=2,3,4$.

\subsection{Performance of the preconditioner in the parametric domain}
The computational space domain is $\Omega = (0,1)^2$ and the space time domain is
$\mathcal{Q}= (0,T) \times \Omega$ with $T=1$. The reference solution is
a traveling wave, that is 
\begin{equation}\label{eq:numerical_solution_traveling_wave}
    u(t,\boldsymbol{x}) = a \exp \left\{-i \dfrac{|\boldsymbol{x}|^2 + t^2} {\omega^2}\right\},
\end{equation}
with wave number $\omega = 0.2$ and amplitude $a = \sqrt[4]{2/\omega^2}$. Here the Dirichlet boundary conditions
are $u|_{\Gamma_D}$. The right hand side is the following
\begin{equation*}
    f(t,\boldsymbol{x}) = a \left( \dfrac{4i}{\omega^2} + \dfrac{4|\boldsymbol{x}|^2}{\omega^4} + \dfrac{2t}{\omega^2} \right) \exp \left\{-i \dfrac{|\boldsymbol{x}|^2 + t^2} {\omega^2}\right\} . 
\end{equation*}
The numerical solution on a mesh of $64$ elements per univariate direction is shown in Figure \ref{fig:traveling_wave_solution} for different time frames. 
\begin{figure}   
  \centering 
  \subfloat[]
  [t=0 sec.]
  {\includegraphics[width = 0.35\textwidth]{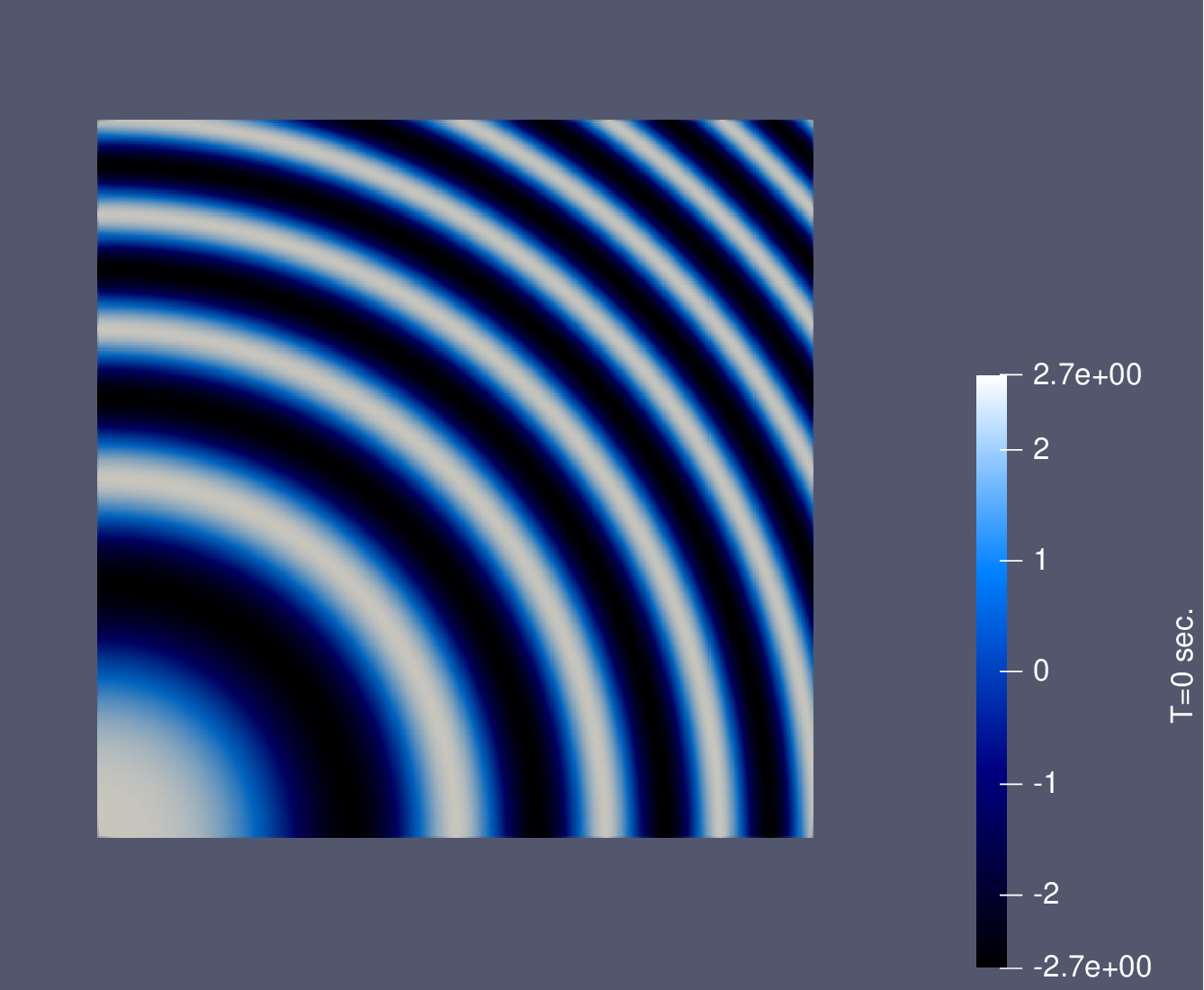}
  \label{figure_a} }
  \subfloat[]
  [t=0,25 sec.]
  {\includegraphics[width = 0.35\textwidth]{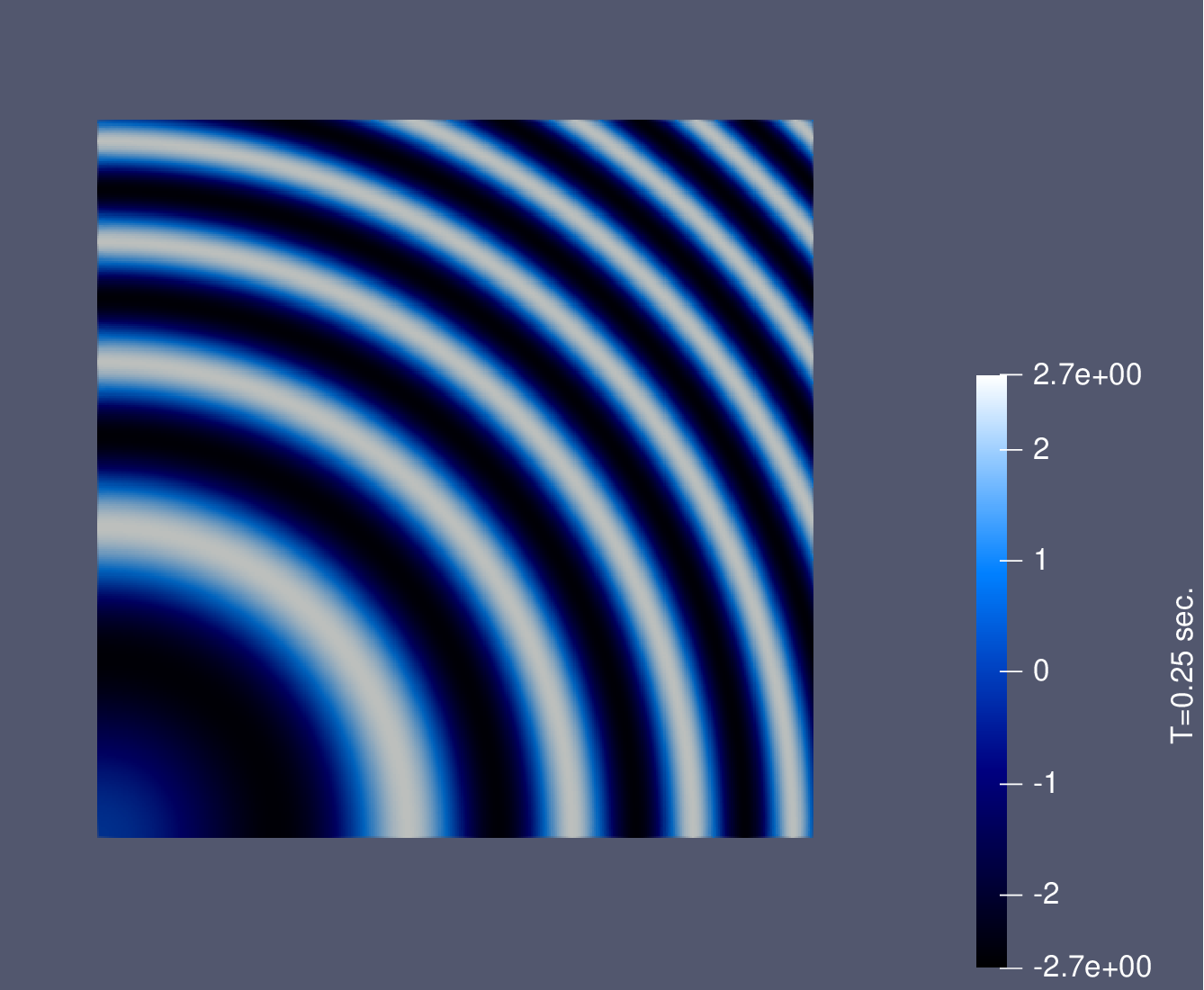} 
   \label{figure_b} }
  \\
  \subfloat[]
  [t=0,5 sec.]
  {\includegraphics[width = 0.35\textwidth]{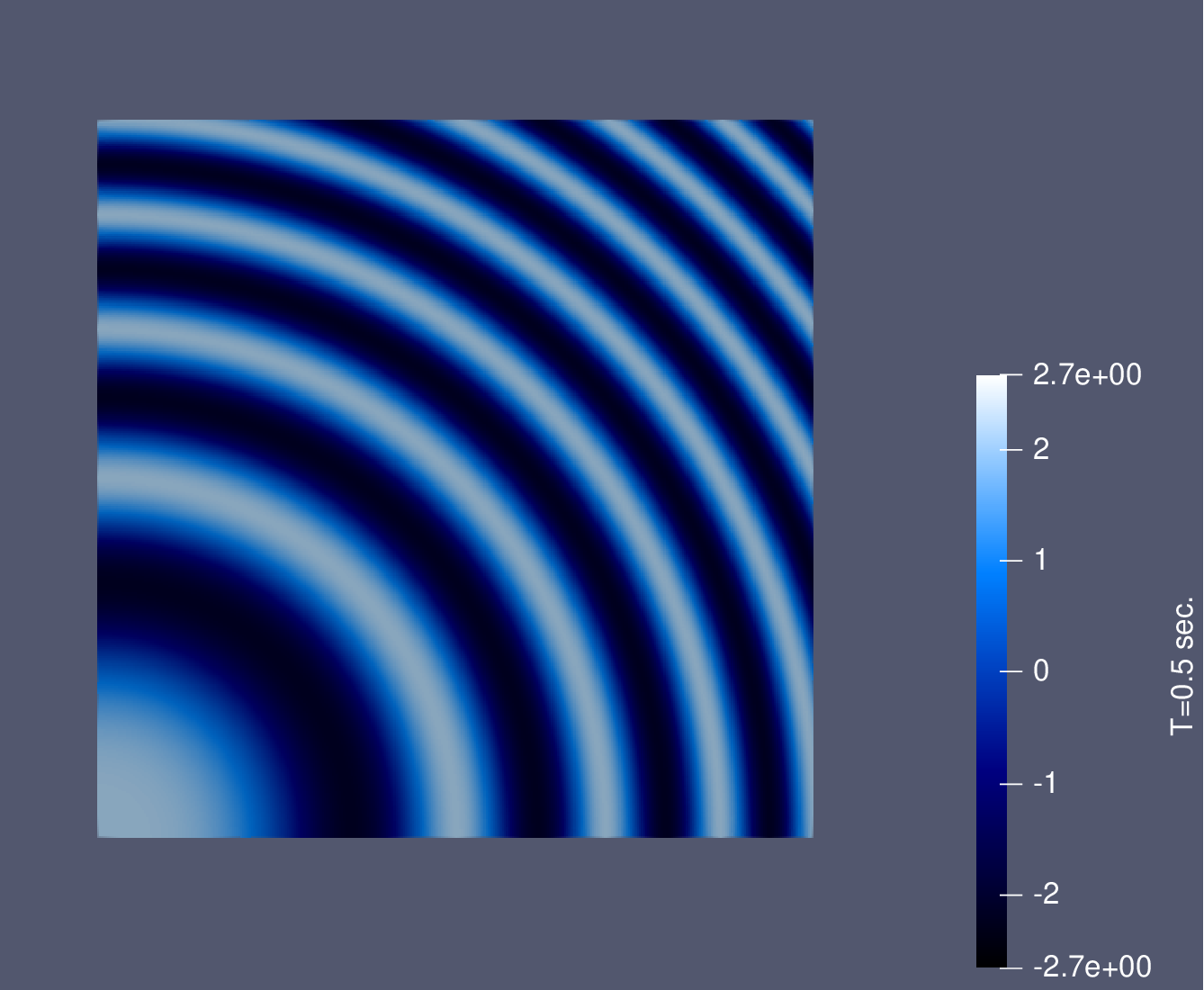}
   \label{figure_c} }   
  \subfloat[]
  [t=0,75 sec.]
  {\includegraphics[width = 0.35\textwidth]{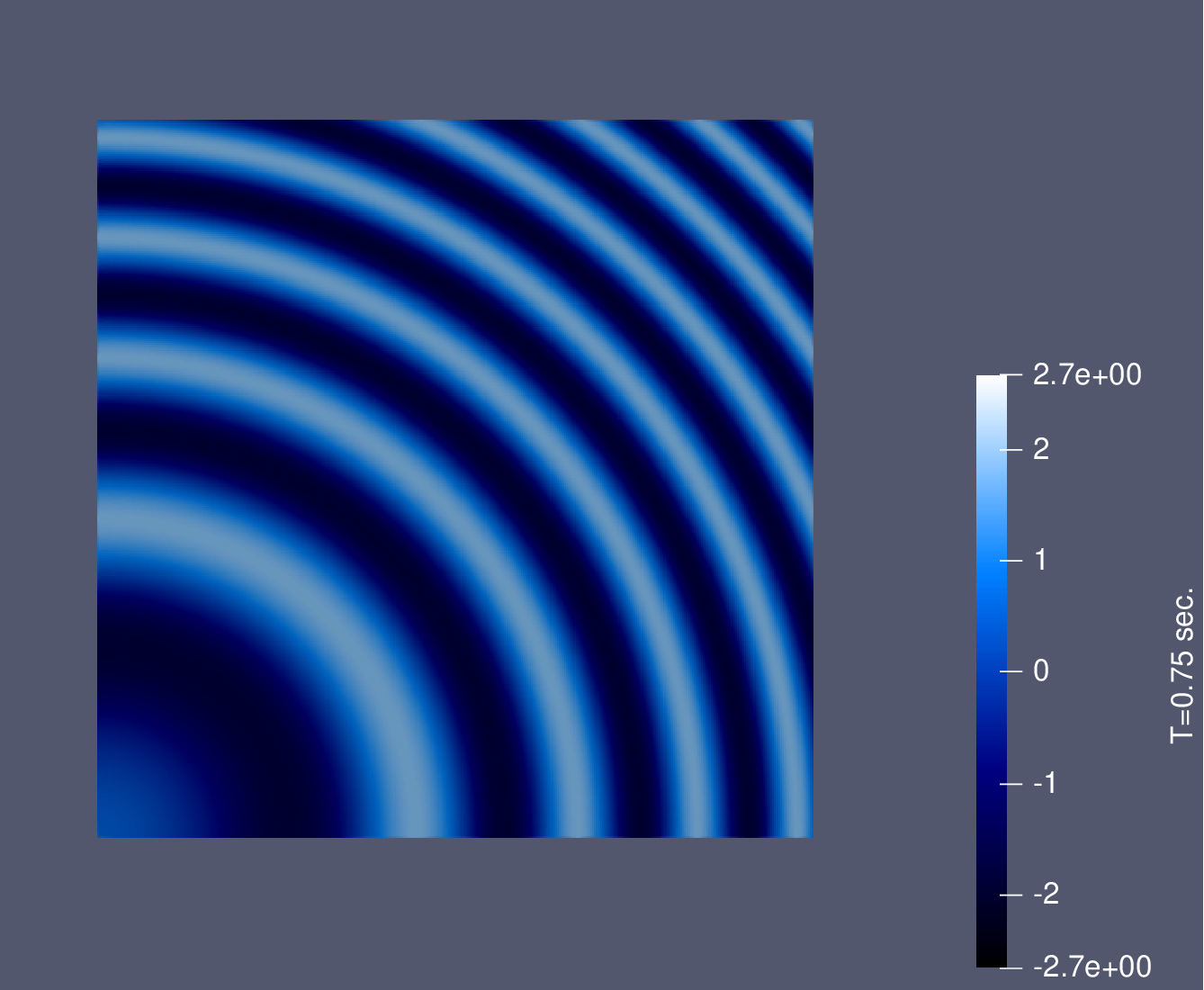}
   \label{figure_d} }
  \caption{Real part of the numerical approximation of \eqref{eq:numerical_solution_traveling_wave} at four different time frames.}
  \label{fig:traveling_wave_solution}
\end{figure}

\begin{table}
    \caption{Parametric domain. Performance of $\widehat{\mathbf{P}}$.}
    \centering
    \label{tab:performance of the preconditioner}
    \begin{tabular}{|r|c|c|c|c|}
        \hline
        \multicolumn{5}{|c|}{Performance of preconditioner} \\
        \hline
        $N_{e}$ & $N_{dof}$ & \texttt{backslash} (time) &$\widehat{\mathbf{P}}$ (iter / time) &  ICHOL (iter / time)\\
        \hline
        \multicolumn{5}{|c|}{Degree $p=2$} \\
        \hline
         8 & 1000   & 0.0306 &7\, / \, 0.0426&  12\, / \,0.0256 \\
        16 & 5832   & 0.7974 &7\, / \, 0.0659&  38\, / \,0.1338 \\
        32 & 39304  & 27.46  &7\, / \, 0.2992& 174\, / \,3.9153 \\
        64 & 287496 & 1148   &7\, / \, 2.7304& $^*$ \\
        \hline
        \multicolumn{5}{|c|}{Degree $p=3$} \\
        \hline
         8 & 1331   &0.0618  & 8\, / \, 0.0302 & 10\, / \,0.0172 \\
        16 & 6859   &1.6415  & 9\, / \, 0.0982 & 32\, / \,0.2884 \\
        32 & 42875  &118     & 8\, / \, 0.4703 &148\, / \,8.1482\, \\
        64 & 300763 &$^{**}$ & 8\, / \, 4.7142 &$^*$ \\
        \hline
        \multicolumn{5}{|c|}{Degree $p=4$} \\
        \hline
         8 & 1728   & 0.1587 & 10\, / \,0.0427  & 10\, / \,0.0548 \\
        16 & 8000   & 2.6433 & 10\, / \,0.2289  & 28\, / \,0.4363 \\
        32 & 46656  & 419    & 10\, / \,1.2329  &127\, / \,14.1742 \\
        64 & 314432 &$^{**}$ & 10\, / \,9.0219 &191\, / \, 176 \, \, \,  \\
        \hline
    \end{tabular}
\end{table}

We analyze the performance of the proposed preconditioner $\widehat{\mathbf{P}}$ for a variety of
uniform partitions up to $n_{el}=64$ per univariate directions, and polynomial degree $p=2,3,4$. 
In Table \ref{tab:performance of the preconditioner}
it is reported the computational clock time cost of solving the linear system both directly
using Matlab backslash and iteratively by CG solver, reporting also the number of iterations
for this latter case. When solving with CG we investigate the performance of the 
solver with preconditioner $\widehat{\mathbf{P}}$, and compare it with a
classical algebraic preconditioner as incomplete Cholesky factorization (ICHOL). 
Matlab backslash is clearly inefficient, since performing gaussian elimination requires $N_{dof}^3$ FLOPs. 
Using classical preconditioners in CG iterative solvers is a reasonable approach for small size problem, 
but the number of iterations grows with the size of the problem.
The performance of the preconditioner $\widehat{\mathbf{P}}$ with CG, 
is identical among $h$-refined meshes, and seams reasonably $p$-robust. 
The number of iterations never exceeds $10$, and the total amount of time required 
to solve the discrete problem, is always cheaper than the other approaches we tested.

\section{Conclusions}
 In this work we proposed and studied a space-time least square method for the Schr{\"o}dinger equation
 in the framework of isogeometric analysis. Our scheme is based on smooth spline in space and time, that allows,
 in the particular case of the parametric domain, to introduce a suitable preconditioner for the arising linear system. 
 Our preconditioner $\widehat{\mathbf{P}}$ is represented by a sum of Kronecker products of matrices, 
 that makes the computational cost of its construction (setup) and application, as well as the storage cost, very appealing.
 In particular the construction of the preconditioner exploits a spectral equivalence between the space matrices $\mathbb{B}_s$
 and $\mathbf{L}_s^T \mathbf{M}_s^{-1} \mathbf{L}_s $ that , thanks to the FD technique, admits a stable block-diagonal factorization.  
 
 The application cost for a serial execution is almost equal to $O(N_{dof})$, and the block-diagonal structure is suitable 
 for parallel implementation on distributed memory machines, and this will be an interesting future direction of study. 

{At the same time, the storage cost is roughly the same that we would have by discretizing separately in space and in
time, if we assume $n_t\leq Cp^dN_s$. Indeed, in this case the memory used for the whole iterative
solver is $O(p^dN_s + N_{dof} )$. Although, our approach could be coupled with a matrix-free idea, and this is expected to further
improve the efficiency of the overall method.}

As a final comment, it would be interesting to further exploits the structure of time pencils, in order to achive a full factorization of the proposed preconditioner.
This may also give a hint in proposing an ad-hoc preconditioner for the isogeometric framework, which we are still working on.

\appendix
\section{Well-posedness of the space-time variational formulation}
Here we extend the results presented in \cite{DEMKOWICZ} on the well posedness of \eqref{eq:var_for}. First we introduce a suitable notation, such that this appendix can be read independently from the paper. 
Let us recall $\mathcal{Q} = (0,T)\times \Omega$,
with $\Omega \subset \mathbb{R}^{d}$, $d=1,2,3$, while $\Gamma_D = (0,T)\times \partial \Omega$. Consider $\Gamma_0 = \Gamma_D \cup (\{0\}\times \Omega)$ 
and $\Gamma_T = \Gamma_D \cup (\{T\}\times \Omega)$ and let us define $\mathcal{D}_0 := \left\{\phi \in C^{\infty}_0(\mathbb{R}^{d+1}) : \phi|_{\Gamma_0} = 0  \right\}$, 
which is the space of smooth functions of $\mathbb{R}^{d+1}$ with compact support such that restricted to $\mathcal{Q}$ satisfy both homogeneous Dirichlet and initial conditions.
Analogously define $\mathcal{D}_T := \left\{\psi \in C^{\infty}_0(\mathbb{R}^{d+1}) : \psi|_{\Gamma_T} = 0  \right\}$, that instead satisfies homogeneous Dirichlet and final conditions. 
Recall $\mathbb{S} := \mathrm{i}\partial_t -\nu \Delta $, and notice that integration by parts gives:
$$
\int_{\Omega}\int_{0}^T  \left(\mathbb{S} \phi \right) \overline{\psi} \,\dt \, \d\Omega = \int_{\Omega}\int_{0}^T \phi \left(\overline{\mathbb{S} \psi} \right) \,\dt \, \d\Omega, \quad \forall \phi \in \mathcal{D}_0, \text{ and } \forall \psi \in \mathcal{D}_T.
$$
The space $\mathcal{V}$ in \eqref{eq:var_for} is the domain of $\mathbb{S}:\mathcal{V}\subset L^2(\mathcal{Q})\to L^2(\mathcal{Q})$, that can be written as:
\begin{equation}\label{eq:def_dom_S}
  \mathcal{V}:=\left\{ v \in L^2(\mathcal{Q}) : \mathbb{S}v \in L^2(\mathcal{Q}) \text{ and } (\mathbb{S}\psi,v) - (\psi,\mathbb{S}v) = 0 
                      \,\,\forall \psi \in \mathcal{D}_T\right\},
\end{equation}
and we have $\mathcal{C}^{\infty}_0(\mathcal{Q}) \subset \mathcal{V} \subset L^2(\mathcal{Q})$, that is $\mathbb{S}$ is densely defined.
Denoting by $\mathbb{S}^*:\mathcal{V}^* \subset L^2(\mathcal{Q})\to L^2(\mathcal{Q})$ the adjoint operator, whose domain is given by 
\begin{equation}\label{eq:def_dom_S*}
  \mathcal{V}^*:=\left\{ w \in L^2(\mathcal{Q}) : \exists g \in L^2(\mathcal{Q}) \text{ such that } (\mathbb{S}v,w) = (v,g) \,\,\forall v \in \mathcal{V}\right\},
\end{equation}
we have $\mathbb{S}^*w := g$, with $\mathbb{S}^* = \mathbb{S}$, and $\mathcal{C}^{\infty}_0(\mathcal{Q}) \subset \mathcal{V}^* \subset L^2(\mathcal{Q})$.
Notice that we are identifying $L^2(\mathcal{Q})' \equiv L^2(\mathcal{Q})$ through Riesz isomorphism.
We endow both $\mathcal{V}$ and $\mathcal{V}^*$ with the norms $\|\cdot\|_{\mathcal{V}}$ and $\|\cdot\|_{\mathcal{V}^*}$ respectively, such that 
\[
  \| v \|_{\mathcal{V}}^2 := \| v \|_{L^2(\mathcal{Q})}^2 + \| \mathbb{S}v \|_{L^2(\mathcal{Q})}^2,  
\quad \text{and} \quad  
  \| v \|_{\mathcal{V}^*}^2 := \| v \|_{L^2(\mathcal{Q})}^2 + \| \mathbb{S}^*v \|_{L^2(\mathcal{Q})}^2.
\]
Define the boundary operators $B:\mathcal{V}\to(\mathcal{V}^*)'$ and $B^*:\mathcal{V}^* \to (\mathcal{V})'$, such that
\begin{subequations}
  \begin{equation}
    \langle Bv,w \rangle := (\mathbb{S}v,w)_{L^2(\mathcal{Q})} - (v,\mathbb{S}^*w)_{L^2(\mathcal{Q})}, 
  \end{equation}    
  \begin{equation}
    \langle B^*w,v \rangle := (\mathbb{S}^*w,v)_{L^2(\mathcal{Q})} - (w,\mathbb{S}v)_{L^2(\mathcal{Q})},
  \end{equation}
\end{subequations}
hold true for all $v,w \in L^2({\mathcal{Q}})$ such that $\mathbb{S}v,\mathbb{S}^*w \in L^2(\mathcal{Q})$. 
From \cite[Lemma A.2]{DEMKOWICZ}, we have 
\begin{equation}\label{eq:ortogonalita_v*_Bv}
  \mathcal{V}^* = (B(\mathcal{V}))^{\bot},
\end{equation}  
and, from \eqref{eq:def_dom_S}, it holds 
\begin{equation}
  \mathcal{V} = (B^*(\mathcal{D}_T))^{\bot}.
\end{equation}  
In particular, from \cite[Lemma 2.1]{DEMKOWICZ}, it holds $\mathcal{D}_0 \subset \mathcal{V}$ and $\mathcal{D}_T \subset \mathcal{V}^*$, and in addition
we make the following density assumption.
\begin{ass}\label{ass:density}
  We assume that $\overline{\mathcal{D}_0}^{\|\cdot\|_{\mathcal{V}}} = \mathcal{V}$ and $\overline{\mathcal{D}_T}^{\|\cdot\|_{\mathcal{V}^*}} = \mathcal{V}^*$.
\end{ass}
Under Assumption \ref{ass:density}, from \cite[Lemma 2.2]{DEMKOWICZ}, it holds 
\begin{equation}\label{eq:caratterizzazione_di_V*}
  \mathcal{V}^* = (B(\mathcal{V}))^{\bot} = (B(\mathcal{D}_0))^{\bot},
\end{equation}  
and 
\begin{equation}\label{eq:caratterizzazione_di_V}
  \mathcal{V} = (B^*(\mathcal{V}^*))^{\bot} = (B^*(\mathcal{D}_T))^{\bot}.
\end{equation}  

The next result states the well posedness of the space time variational formulation. 
\begin{thm}
  Under Assumption \ref{ass:density}, the linear Schr{\"o}dinger operator $\mathbb{S}: \mathcal{V}\to L^2(\mathcal{Q})$ is a continuous bijections, 
  that is problem \eqref{eq:var_for} is well posed.
\end{thm}
The proof of well posedness is given in \cite[Theorem 2.4]{DEMKOWICZ}, and Assumption \ref{ass:density}, but more precisely \eqref{eq:caratterizzazione_di_V*}, 
is needed only to prove injectivity of $\mathbb{S}$. 
We now prove that, Assumption \eqref{ass:density} is verified for every integer $d\geq 1$. 
\begin{lem}
  Given $\mathcal{Q}=(0,T) \times \Omega $, with $\Omega = [0,1]^d$ and integer $d\geq 1$, then Assumption \eqref{ass:density} holds true.
\end{lem}
\begin{proof}
 We prove that $\mathcal{D}_0 $ is dense in $\mathcal{V}$, the other stated density result is analogous. The case $d=1$ is in \cite[Theorem 3.1]{DEMKOWICZ}, 
 thus we fix an integer $d>1$. Consider $v\in \mathcal{V}$, first we extend $v$ to the whole $\mathbb{R}^{d+1}$ domain. 
 \begin{enumerate}
  \item \textit{Extending along space directions:}
        Let us extend the space-time domain among the space directions as follows.
        Denote by $\mathcal{Q}_{1,c}= \mathcal{Q}$, 
        $\mathcal{Q}_{1,l} := [0,T]\times [-1,0] \times [0,1]^{d-1}$ and  
        $\mathcal{Q}_{1,r} := [0,T]\times [1,2]  \times [0,1]^{d-1}$. 
        Analogously, for $i=2,\dots,d$, we introduce 
        $\mathcal{Q}_{i,c} := \bigcup_{j \in \{l,c,r\}} \mathcal{Q}_{i-1,j}$, then
        $\mathcal{Q}_{i,l} := [0,T]\times [-1,2]^{i-1} \times [-1,0] \times [0,1]^{d-i}$ and  
        $\mathcal{Q}_{i,r} := [0,T]\times [-1,2]^{i-1} \times [1,2]  \times [0,1]^{d-i}$, considering $[0,1]^{0}= \emptyset$. 
        Finally let us call $\Omega_{E}:= \bigcup_{j \in \{l,c,r\}} \mathcal{Q}_{d,j}$ the enlarged space-time cylinder.
        Then, we introduce the intermediate extension operators,  $E_i : \mathcal{Q}_{i,c} \to \mathcal{Q}_{i+1,c}$ for $i=1,\dots,d-1$,
        and $ E_d : \mathcal{Q}_{d,c} \to \mathcal{Q}_{E}$, such that  
        \[ 
          E_i f (t,\boldsymbol{x}) :=
          \begin{cases}
            -f(t,x_1,\dots,-x_i,\dots,x_d)  &(t,\boldsymbol{x}) \in \mathcal{Q}_{i,l},\\
             f(t,x_1,\dots,x_i,\dots,x_d)   &(t,\boldsymbol{x}) \in \mathcal{Q}_{i,c},\\
            -f(t,x_1,\dots,2-x_i,\dots,x_d) &(t,\boldsymbol{x}) \in \mathcal{Q}_{i,r}.
          \end{cases} 
        \]
        We denote the reverse operators by $E_i'$, defined as 
        $E_i'g(t,\boldsymbol{x}) = g(t,\boldsymbol{x})-g(t,x_1,\dots,-x_i,\dots,x_d)-g(t,x_1,\dots,2-x_i,\dots,x_d) $,
        for $(t,\boldsymbol{x})\in\mathcal{Q}_{i,c}$, and for $i = 1,\dots,d$. 
        The definitions are to be interpreted almost everywere, and finally our extension operator from $\mathcal{Q}$ to $\mathcal{Q}_E$ is 
        $E:= E_d \circ \dots, \circ E_1$, while its reverse operator from $\mathcal{Q}_E$ to $\mathcal{Q}$ is $E' = E_1'\circ \dots  \circ E_d'$.
        It is easy to see by a change of variable that 
        $$
          (Ef,g)_{L^2(\mathcal{Q}_E)} = (f,E'g)_{L^2(\mathcal{Q})}, \quad \forall f \in L^2(\mathcal{Q}), \forall g \in L^2(\mathcal{Q}_E).
        $$
        Next, we claim that 
        $$ 
          \mathbb{S}E v = E\mathbb{S} v ,\quad  \forall v \in \mathcal{V}. 
        $$
        Clearly, $Ev$ is in $L^2(\mathcal{Q}_E)$ and notice that $E'\mathbb{S}\varphi = \mathbb{S}E'\varphi$ for all $ \varphi \in \mathcal{C}^{\infty}_0(\mathcal{Q}_E)$. Therefore, it holds 
        \[\begin{split}
        \langle \mathbb{S}Ev, \varphi \rangle_{\mathcal{C}^{\infty}_0(\mathcal{Q}_E)} 
        &= (Ev,\mathbb{S}\varphi)_{L^2(\mathcal{Q}_E)} 
        = (v,E'\mathbb{S}\varphi)_{L^2(\mathcal{Q})} =\\
        &= (v,\mathbb{S}E'\varphi)_{L^2(\mathcal{Q})}
        = (\mathbb{S}v,E'\varphi)_{L^2(\mathcal{Q})} - \langle Bv, E'\varphi\rangle.
        \end{split}
        \]
        Now, since $E'\varphi|_{\Gamma_T}=0$, we have $E'\varphi \in \mathcal{V}^*$, and 
        thus $\langle Bv, E'\varphi\rangle =0$ by \eqref{eq:ortogonalita_v*_Bv}. 
        It follows that 
        $$
        \langle \mathbb{S}Ev, \varphi \rangle_{\mathcal{C}^{\infty}_0(\mathcal{Q}_E)} 
        = (\mathbb{S}v,E'\varphi)_{L^2(\mathcal{Q})} 
        = (E\mathbb{S}v,E\varphi)_{L^2(\mathcal{Q}_E)}, 
        $$
        completing the proof of the claim. We also conclude that $\mathbb{S}Ev $ is in $L^2(\mathcal{Q}_E)$ whenever $v\in \mathcal{V}$.
  \item \textit{Extending along time direction:}
        Let $\widetilde{E}$ denote the extension of $E$ by zero to $\mathbb{R}^{d+1}$, and $\tau_{\delta}$ be the translation operator in $t$ direction, i.e.,
        $\tau_{\delta}w(t,\boldsymbol{x}) = w(t-\delta, \boldsymbol{x})$.
        From \cite{brezis2011functional} it holds 
        \begin{equation}\label{eq:limite_delta}
          \lim_{\delta \to 0} \|\tau_{\delta}w-w\|_{L^2(\mathbb{R}^{d+1})}=0, \quad \forall w \in L^2({\mathbb{R}^{d+1}}).
        \end{equation}
        Introducing $\mathcal{Q}_{E,\delta} = (-\delta,T+\delta) \times (-1,2)^{d}$, by a change of variables, it holds 
        $$
          (\tau_{\delta}\widetilde{E}f,g)_{L^2(\mathcal{Q}_{E,\delta})} = (Ef,\tau_{-\delta}g)_{L^2(\mathcal{Q}_{E})},\quad \forall f \in L^2(\mathcal{Q}), \forall g \in L^2(\mathcal{Q}_{E,\delta}). 
        $$
        Denoting by $R_{\delta}$ the restriction operator of function on $\mathbb{R}^{d+1}$ to $\mathcal{Q}_{E,\delta}$, we now claim 
        that 
        $$
        \mathbb{S}R_{\delta}\tau_{\delta}\widetilde{E}v = R_{\delta}\tau_{\delta}\widetilde{E}\mathbb{S}v, \quad \forall v \in \mathcal{V}.
        $$
        The proof is analogous to the one in Step 1. Given $\varphi \in \mathcal{C}^{\infty}_0(\mathcal{Q}_{E,\delta})$ it holds 
        \begin{equation*}
          \begin{split}
            \langle \mathbb{S}R_{\delta}\tau_{\delta}\widetilde{E}v, \varphi \rangle_{\mathcal{C}^{\infty}_0(\mathcal{Q}_{E,\delta})} 
            &= (\tau_{\delta}\widetilde{E}v,\mathbb{S}\varphi)_{L^2(\mathcal{Q}_{E,\delta})} 
             = (Ev,\mathbb{S}\tau_{-\delta}\varphi)_{L^2(\mathcal{Q}_E)}\\
            &= (v,E'\mathbb{S}\tau_{-\delta}\varphi)_{L^2(\mathcal{Q})}
             = (v,\mathbb{S}E'\tau_{-\delta}\varphi)_{L^2(\mathcal{Q})}\\
            &= (\mathbb{S}v,E'\tau_{-\delta}\varphi)_{L^2(\mathcal{Q})} - \langle Bv, E'\tau_{-\delta}\varphi\rangle.
          \end{split}
        \end{equation*}
        Now, since $E'\tau_{-\delta}\varphi|_{\Gamma_T}=0$, we have $E'\tau_{-\delta}\varphi \in \mathcal{V}^*$, and 
        thus $\langle Bv, E'\tau_{-\delta}\varphi\rangle =0$ by \eqref{eq:ortogonalita_v*_Bv}. It follows that 
        $$
        \langle \mathbb{S}R_{\delta}\tau_{\delta}\widetilde{E}v, \varphi \rangle_{\mathcal{C}^{\infty}_0(\mathcal{Q}_{E,\delta})} 
        = (\widetilde{E}\mathbb{S}v,\tau_{-\delta}\varphi)_{L^2(\mathcal{Q}_E)}
        = (\tau_{-\delta}\widetilde{E}\mathbb{S}v,\varphi)_{L^2(\mathcal{Q}_{E,\delta})},
        $$ 
        which proves the claim.
  \item \textit{Mollify:}
        Consider the mollifier $\rho_{\epsilon} \in \mathcal{C}^{\infty}_0(\mathbb{R}^{d+1})$, defined by 
        $$\rho_{\epsilon}(t,\boldsymbol{x}):= \epsilon^{-d-1}\rho_1(\epsilon^{-1}t,\epsilon^{-1}x_1,\dots,\epsilon^{-1}x_d),\quad \text{for} \epsilon >0,$$
        where
        $$
        \rho_1 (t,\boldsymbol{x}) := 
        \begin{cases}
          k e^{-1/{(1-|(t,\boldsymbol{x})|^2)}}, & \text{if } |(t,\boldsymbol{x})|^2 < 1, \\
          0 & \text{if } |(t,\boldsymbol{x})|^2 \geq 1,
        \end{cases}
        $$
        with $|\cdot|$ denoting the Euclidean norm in $\mathbb{R}^{d+1}$, and $k$ is a constant chosen such that $\int_{\mathbb{R}^{d+1}}\rho_1=1$.
        Notice that, given $\delta>0$ small enough, i.e., $\delta<\min\{T/2,1/2\}$, the convolutions $v_{\epsilon}:=\rho_{\epsilon}\ast\tau_{\delta}\widetilde{E}v$ 
        and $s_{\epsilon}:= \rho_{\epsilon}\ast\tau_{\delta}\widetilde{E}\mathbb{S}v $ are smooth functions that satisfy 
        \begin{equation}\label{eq:limite_epsilon}
          \lim_{\epsilon\to0}\|v_{\epsilon} - \tau_{\delta}\widetilde{E}v\|_{L^2(\mathbb{R}^{d+1})} = 0,\quad \text{and} \quad  \lim_{\epsilon\to0}\|s_{\epsilon} - \tau_{\delta}\widetilde{E}\mathbb{S}v\|_{L^2(\mathbb{R}^{d+1})} = 0.
        \end{equation}
        Moreover, the smooth function $\mathbb{S}v_{\epsilon}$ need not coincide to $s_{\epsilon}$ everywere, but they coincide on $\mathcal{Q}$ whenever $\epsilon<\delta/2$. 
        Thus, consider $\delta = 3\epsilon$, and let $\epsilon<\min\{T/6,1/6\}$ go to zero. We have 
        \begin{subequations}
        \begin{equation*}
            \|\mathbb{S}v_{\epsilon}-\mathbb{S}v\|_{L^2(\mathcal{Q})}
            = \|s_{\epsilon}-\mathbb{S}v\|_{L^2(\mathcal{Q})}  
            \leq \|s_{\epsilon}-\tau_{\delta}\widetilde{E}\mathbb{S}v\|_{L^2(\mathbb{R}^{d+1})} + \|\tau_{\delta}\widetilde{E}\mathbb{S}v-\widetilde{E}\mathbb{S}v\|_{L^2(\mathbb{R}^{d+1})}        
        \end{equation*}
        and 
        \begin{equation*}
            \|v_{\epsilon}-v\|_{L^2(\mathcal{Q})}
            \leq \|v_{\epsilon}-\tau_{\delta}\widetilde{E}v\|_{L^2(\mathcal{Q})} + \|\tau_{\delta}\widetilde{E}v-\widetilde{E}v\|_{L^2(\mathcal{Q})}        
        \end{equation*}
        \end{subequations}
        Using \eqref{eq:limite_delta} and \eqref{eq:limite_epsilon}, it follows that  
        $$
          \lim_{\epsilon \to 0} \|v_{\epsilon}-v\|_{\mathcal{V}} = 0.
        $$
        To conclude, we examine the value of $v_{\epsilon}$ at the edges of the space-time cylinder.
        We have 
        $$
        \strain{
        v_{\epsilon}(t,0,x_2,\dots,x_d) = \int_{\mathbb{R}^{d+1}} \rho_{\epsilon}(t-\sigma,-r_1,x_2-r_2,\dots,x_d-r_d) \tau_{\delta}\widetilde{E}v(\sigma,r_1,\dots,r_d) \d r_1\dots\d r_d\d\sigma,
        }
        $$
        with the integrand the inner integral being the product of and even function $\rho_{\epsilon}$, with respect to $r_1$,
        and an odd function $\tau_{\delta}\widetilde{E}v$ of $r_1$. Thus $v_{\epsilon}(t,0,x_2,\dots,x_d)=0$ and the same holds for $v_{\epsilon}(t,1,x_2,\dots,x_d)$ and the other univariate space dierctions. Moreover
        since $\tau_{\delta}\widetilde{E}v$ is identically zero in a neighborhood of $(0,\boldsymbol{x})$, we conclude that $v_{\epsilon}|_{\Gamma_0} =0$.
 \end{enumerate}

\end{proof}
Next we extend this result to the isogeometric case domain.
\begin{thm}
  Given $\mathcal{Q}=(0,T) \times \Omega $, with $\Omega = \mathbf{F}([0,1]^d)$ and integer $d\geq 1$, then Assumption \eqref{ass:density} holds true.
\end{thm}
\begin{proof}
  We prove that $\mathcal{D}_0 $ is dense in $\mathcal{V}$, the other stated density result is analogous.   
  Given $v\in \mathcal{V}$, recall $\mathbf{G}:[0,1]^{d+1} \to \mathcal{Q}$ is the parameterization of the space-time cilinder, such that $\mathbf{G}(\tau,\boldsymbol{\eta}):= (T\tau,\mathbf{F}(\boldsymbol{\eta}))=(t,\boldsymbol{x})$.
  Define $\widehat{v}:= v\circ  \mathbf{G}$. Clearly $\widehat{v} \in L^2([0,1]^{d+1})$, and $\mathbb{S}\widehat{v}\in L^2([0,1]^{d+1})$. Moreover $(\mathbb{S}v,\phi)_{L^2([0,1]^{d+1})} - (v,\mathbb{S}\phi)_{L^2([0,1]^{d+1})} = 0 $, 
  for all $\phi \in \mathcal{C}^{\infty}_0(\mathbb{R}^{d+1})$ such that $\phi|_{\mathbf{G}^{-1}(\Gamma_T)} = 0 $. 
  By applying Lemma 1, it exists $\{\widehat{v}_{\epsilon}\}_{\epsilon >0} \subset \mathcal{C}^{\infty}_0(\mathbb{R}^{d+1})$ such that $\widehat{v}_{\epsilon}|_{\mathbf{G}^{-1}(\Gamma_0)} = 0$, that satisfies 
  $$
  \lim_{\epsilon \to 0} \| \widehat{v}_{\epsilon} - \widehat{v} \|^2_{L^2([0,1]^{d+1})} + \| \mathbb{S}\widehat{v}_{\epsilon} - \mathbb{S}\widehat{v} \|^2_{L^2([0,1]^{d+1})} = 0.
  $$
  Therefore define $v_{\epsilon}:= \widehat{v}_{\epsilon} \circ \mathbf{G}^{-1}$ and notice that $\{v_{\epsilon}\}_{\epsilon>0} \subset \mathcal{D}_0$. Moreover, 
  by a change of variable, 
  $$
  \lim_{\epsilon \to 0} \| {v}_{\epsilon} - {v} \|^2_{\mathcal{V}} \leq C \lim_{\epsilon \to 0} \| \widehat{v}_{\epsilon} - \widehat{v} \|^2_{L^2([0,1]^{d+1})} + \| \mathbb{S}\widehat{v}_{\epsilon} - \mathbb{S}\widehat{v} \|^2_{L^2([0,1]^{d+1})} = 0.
  $$
  This completes the proof.
\end{proof}

\section*{Acknowledgments}

\bibliography{biblio_space_time}

\begin{thebibliography}{10}
\expandafter\ifx\csname url\endcsname\relax
  \def\url#1{\texttt{#1}}\fi
\expandafter\ifx\csname urlprefix\endcsname\relax\def\urlprefix{URL }\fi
\expandafter\ifx\csname href\endcsname\relax
  \def\href#1#2{#2} \def\path#1{#1}\fi

\bibitem{fried1969finite}
I.~Fried, Finite-element analysis of time-dependent phenomena., AIAA Journal
  7~(6) (1969) 1170--1173.

\bibitem{bruch1974transient}
J.~C. Bruch~Jr., G.~Zyvoloski, Transient two-dimensional heat conduction
  problems solved by the finite element method, International Journal for
  Numerical Methods in Engineering 8~(3) (1974) 481--494.

\bibitem{oden1969general}
J.~T. Oden, A general theory of finite elements. {I}. {T}opological
  considerations, International Journal for Numerical Methods in Engineering
  1~(2) (1969) 205--221.

\bibitem{oden1969general2}
J.~T. Oden, A general theory of finite elements. {II}. {A}pplications,
  International Journal for Numerical Methods in Engineering 1~(3) (1969)
  247--259.

\bibitem{shakib1991new}
F.~Shakib, T.~J.~R. Hughes, A new finite element formulation for computational
  fluid dynamics: {IX}. {F}ourier analysis of space-time
  {G}alerkin/least-squares algorithms, Computer Methods in Applied Mechanics
  and Engineering 87~(1) (1991) 35--58.

\bibitem{karakashian1998space}
O.~Karakashian, C.~Makridakis, A space-time finite element method for the
  nonlinear {S}chr{\"o}dinger equation: the discontinuous {G}alerkin method,
  Mathematics of computation 67~(222) (1998) 479--499.

\bibitem{DEMKOWICZ}
L.~Demkowicz, J.~Gopalakrishnan, S.~Nagaraj, P.~Sepulveda, A spacetime {DPG}
  method for the {S}chr{\"o}dinger equation, SIAM Journal on Numerical Analysis
  55~(4) (2017) 1740--1759.

\bibitem{gomez2022space}
S.~G{\'o}mez, A.~Moiola, A space-time {T}refftz discontinuous {G}alerkin method
  for the linear {S}chr{\"o}dinger equation, SIAM Journal on Numerical Analysis
  60~(2) (2022) 688--714.

\bibitem{hain2022ultra}
S.~Hain, K.~Urban, An ultra-weak space-time variational formulation for the
  {S}chr{\"o}dinger equation, arXiv preprint arXiv:2212.14398 (2022).

\bibitem{Hughes2005}
T.~J.~R. Hughes, J.~A. Cottrell, Y.~Bazilevs, Isogeometric analysis: {CAD},
  finite elements, {NURBS}, exact geometry and mesh refinement, Computer
  Methods in Applied Mechanics and Engineering 194~(39) (2005) 4135--4195.

\bibitem{Cottrell2009}
J.~A. Cottrell, T.~J.~R. Hughes, Y.~Bazilevs, Isogeometric analysis: toward
  integration of {CAD} and {FEA}, John Wiley \& Sons, 2009.

\bibitem{Evans_Bazilevs_Babuska_Hughes}
J.~A. Evans, Y.~Bazilevs, I.~Babu\v{s}ka, T.~J.~R. Hughes, $n$-widths,
  sup-infs, and optimality ratios for the $k$-version of the isogeometic finite
  element method, Computer Methods in Applied Mechanics and Engineering 198
  (2009) 1726--1741.

\bibitem{bressan2018approximation}
A.~Bressan, E.~Sande, Approximation in {FEM}, {DG} and {IGA}: a theoretical
  comparison, Numerische Mathematik (2019).

\bibitem{Sangalli2018}
G.~Sangalli, M.~Tani, Matrix-free weighted quadrature for a computationally
  efficient isogeometric $k$-method, Computer Methods in Applied Mechanics and
  Engineering 338 (2018) 117 -- 133.

\bibitem{Montardini2018space}
M.~Montardini, M.~Negri, G.~Sangalli, M.~Tani, {S}pace-time least-squares
  isogeometric method and efficient solver for parabolic problems, Mathematics
  of Computation (accepted for publication) (2019).

\bibitem{loli2020efficient}
G.~Loli, M.~Montardini, G.~Sangalli, M.~Tani, An efficient solver for
  space--time isogeometric {G}alerkin methods for parabolic problems, Computers
  \& Mathematics with Applications 80~(11) (2020) 2586--2603.

\bibitem{Lynch1964}
R.~E. Lynch, J.~R. Rice, D.~H. Thomas, Direct solution of partial difference
  equations by tensor product methods, Numerische Mathematik 6~(1) (1964)
  185--199.

\bibitem{dorao2007parallel}
C.~A. Dorao, H.~A. Jakobsen, A parallel time--space least-squares spectral
  element solver for incompressible flow problems, Applied Mathematics and
  Computation 185~(1) (2007) 45--58.

\bibitem{Gander2015}
M.~J. Gander, 50 years of time parallel time integration, in: Multiple Shooting
  and Time Domain Decomposition Methods, Springer, 2015, pp. 69--113.

\bibitem{kvarving2011fast}
A.~M. Kvarving, E.~M. R{\o}nquist, A fast tensor-product solver for
  incompressible fluid flow in partially deformed three-dimensional domains:
  {P}arallel implementation, Computers \& Fluids 52 (2011) 22--32.

\bibitem{Evans2010book}
L.~C. Evans, {Partial Differential equations}, American Mathematical Society,
  Berlin, 2010.

\bibitem{DeBoor2001}
C.~De~Boor, {A practical guide to splines (revised edition)}, Applied
  {M}athematical {S}ciences, Springer, Berlin, 2001.

\bibitem{Kolda2009}
T.~G. Kolda, B.~W. Bader, Tensor decompositions and applications, SIAM review
  51~(3) (2009) 455--500.

\bibitem{Da2012}
L.~Beir{\~a}o~da Veiga, D.~Cho, G.~Sangalli, Anisotropic {NURBS} approximation
  in isogeometric analysis, Computer Methods in Applied Mechanics and
  Engineering 209 (2012) 1--11.

\bibitem{Deville2002}
M.~O. Deville, P.~F. Fischer, E.~H. Mund, High-order methods for incompressible
  fluid flow, Cambridge University Press, 2002.

\bibitem{gahalaut2014condition}
K.~P.~S. Gahalaut, S.~K. Tomar, C.~Douglas, Condition number estimates for
  matrices arising in {NURBS} based isogeometric discretizations of elliptic
  partial differential equations, arXiv preprint arXiv:1406.6808 (2014).

\bibitem{henning2022ultraweak}
J.~Henning, D.~Palitta, V.~Simoncini, K.~Urban, An ultraweak space-time
  variational formulation for the wave equation: Analysis and efficient
  numerical solution, ESAIM: Mathematical Modelling and Numerical Analysis
  56~(4) (2022) 1173--1198.

\bibitem{Sangalli2016}
G.~Sangalli, M.~Tani, Isogeometric preconditioners based on fast solvers for
  the {S}ylvester equation, SIAM Journal on Scientific Computing 38~(6) (2016)
  A3644--A3671.

\bibitem{Montardini2018}
M.~Montardini, G.~Sangalli, M.~Tani, Robust isogeometric preconditioners for
  the {S}tokes system based on the {F}ast {D}iagonalization method, Computer
  Methods in Applied Mechanics and Engineering 338 (2018) 162 -- 185.

\bibitem{Vazquez2016}
R.~V{\'a}zquez, A new design for the implementation of isogeometric analysis in
  {O}ctave and {M}atlab: {G}eo{PDE}s 3.0, Computers \& Mathematics with
  Applications 72~(3) (2016) 523--554.

\bibitem{Sorber2014}
L.~Sorber, M.~Van~Barel, L.~De~Lathauwer, Tensorlab v2. 0, Available online,
  URL: www.tensorlab.net (2014).

\bibitem{brezis2011functional}
H.~Brezis, Functional analysis, {S}obolev spaces and partial differential
  equations, Vol.~2, Springer, 2011.

\end{thebibliography}
\end{document}